\documentclass[12pt]{amsart}
\usepackage{degenerations}

\title{Degenerations of Complex Dynamical Systems II: Analytic and Algebraic Stability}

\author{Laura DeMarco and Xander Faber, with an appendix by Jan Kiwi}
 

\begin{document}
\begin{abstract}
We study pairs $(f, \Gamma)$ consisting of a non-Archimedean rational function $f$ and a finite set of vertices $\Gamma$ in the Berkovich projective line, under a certain stability hypothesis.  We prove that stability can always be attained by enlarging the vertex set $\Gamma$.  As a byproduct, we deduce that meromorphic maps preserving the fibers of a rationally-fibered complex surface are algebraically stable after a proper modification.  The first article in this series examined the limit of the equilibrium measures for a degenerating $1$-parameter family of rational functions on the Riemann sphere.  Here we construct a convergent countable-state Markov chain that computes the limit measure.  

A classification of the periodic Fatou components for non-Archimedean rational functions, due to Rivera-Letelier, plays a key role in the proofs of our main theorems.  The appendix contains a proof of this classification for all tame rational functions.
\end{abstract}

\date{\today}

\maketitle


\section{Introduction}

In the preceding article \cite{DeMarco_Faber_Degenerations_2014}, we studied the dynamics in a degenerating 1-parameter family $f_t$ of complex rational functions by analyzing an associated rational function $f$ on the Berkovich projective line $\Berk$.  Our main tool was a quantization technique: the family $f_t$ gave rise to a sequence of partitions of $\Berk$, and the dynamics of $f$ relative to these partitions provided information on the limit of the measures of maximal entropy for $f_t$ at the degenerate parameter.   In the present article, we refine this quantization technique by constructing a discrete dynamical system that allows us to compute the limiting measures.  Though written as a continuation of \cite{DeMarco_Faber_Degenerations_2014}, with the theorems named accordingly, this article may be read independently.  

Let $k$ be an algebraically closed field that is complete with respect to a nontrivial non-Archimedean valuation.  Let $f$ be a rational function of degree $d \geq 2$ with coefficients in $k$.  Recall that $f$ determines a unique probability measure $\mu_f$ on the Berkovich space $\Berk_k$, characterized by its invariance under $\frac{1}{d} f^*$ while not charging exceptional points for $f$ \cite{Favre_Rivera-Letelier_Ergodic_2010}; we refer to $\mu_f$ as the equilibrium measure for $f$.
  
Let $\Gamma$ be any finite set of type~II points in $\Berk_k$.  
In this article, we introduce a notion of analytic stability for pairs $(f, \Gamma)$; see Definition~\ref{Def: Stable}.  Assuming stability, we show that the equilibrium measure for $f$, as ``seen" from $\Gamma$, is the stationary distribution for an explicit countable-state Markov chain.   (See, e.g., \cite{KSK-Markov_Chains} for definitions and relevant background for this type of random process.) The analytic stability hypothesis parallels the notion of algebraic stability for dynamics on complex surfaces; we obtain a corollary on the existence of algebraically stable modifications for meromorphic maps preserving the fibers of a rational fibration.  

To state our results precisely, we define a partition of $\Berk_k$ consisting of the elements of $\Gamma$ and the connected components of the complement $\Berk_k \smallsetminus \Gamma$.  Analytic stability implies that there is an explicit countable subset $\JJ$ of the partition that contains the Julia set for $f$ and satisfies the following property: For any pair $U, V \in \JJ$, the number of pre-images $\#\left(f^{-1}(y) \cap V \right)$ is independent of $y \in U$, when counted with multiplicities. To each pair of subsets $U,V \in \JJ$, we define a quantity $P_{U,V} \in [0,1]$ by 
	\[
		P_{U,V} = \frac{\#\left(f^{-1}(y) \cap V \right)}{d} \qquad (y \in U).
	\]
Let $P$ be the $|\JJ| \times |\JJ|$ matrix with $(U,V)$-entry $P_{U,V}$. Recall that a stationary probability vector for $P$ is a row vector $\nu: \JJ \to [0,1]$ whose entries sum to~1 and which satisfies $\nu P = \nu$. 


\begin{thmC}
	Let $k$, $f$, $\Gamma$ be as above. Suppose that the pair $(f, \Gamma)$ is analytically stable and that the Julia set for $f$ is not contained in the finite set $\Gamma$. Then $P$ is the transition matrix for a countable state Markov chain with a unique stationary probability vector $\nu: \JJ \to [0,1]$. The rows of $P^n$ converge pointwise to $\nu$, and the $U$-entry of $\nu$ satisfies $\nu(U) = \mu_f(U)$ for each $U \in \JJ$, where $\mu_f$ is the equilibrium measure for~$f$. 
\end{thmC}

\noindent
We provide example computations of $P$ and the stationary measure $\nu$ in Section \ref{Sec: Examples}, and we explain how Theorem~C may be viewed as a generalization of the methods in \cite{DeMarco_Boundary_Maps_2005} for degenerating families of complex rational maps.   

\begin{remark}
	The hypothesis that the Julia set is not contained in $\Gamma$ is necessary; see Example~\ref{Ex: Simple reduction}. However, it fails if and only if there is a totally $f$-invariant point $\zeta \in \Gamma$, in which case $\mu_f$ is supported at $\zeta$. This condition is easily verified in practice. 
\end{remark}

\begin{remark}  
The transition matrix $P$ is a discretized version of the pullback operator $\frac{1}{d}f^*$ acting on probability measures on $\Berk$.   
\end{remark}

	
The convergence of the Markov chain $P$ is essentially a combinatorial reformulation of the equidistribution result of Favre and Rivera-Letelier on iterated pre-images \cite{Favre_Rivera-Letelier_Ergodic_2010}.  It is remarkable that such a combinatorial description of the dynamics of $f$ exists at all, and the analytic stability hypothesis is critical in this respect.  However, it is not a major restriction for many applications. We show that any vertex set $\Gamma$ may be augmented to be analytically stable, under a suitable hypothesis on the field of definition of $f$. 
	
\begin{thmD}
	Let $\ell$ be a discretely valued field with residue characteristic zero, and let $f \in \ell(z)$ be a rational function of degree $d \geq 2$. Let $k$ be a minimal complete and algebraically closed non-Archimedean extension of $\ell$. For any vertex set $\Gamma$ in $\Berk_k$, there exists a vertex set $\Gamma'$ containing $\Gamma$ such that $(f, \Gamma')$ is analytically stable. Moreover, if every element of $\Gamma$ is $\ell$-split (Definition~\ref{Def: ell-split}), then one may take $\Gamma'$ to have the same property. 
\end{thmD}

\begin{remark}
\label{Rem: Tame extension}
Following Trucco \cite{Trucco_Tame_Polynomials_2014}, we say that a rational function $f$ is \textbf{tame} if its ramification locus is contained in the connected hull of the (type~I) critical points. For example, if the residue characteristic of $k$ is zero, or if the residue characteristic is $p > \deg(f)$, then $f$ is tame. We expect that Theorem~D continues to hold without the residue characteristic zero hypothesis, provided that we assume $f$ is tame. 
\end{remark}

A key ingredient in the proof of Theorem~D is a classification of periodic Fatou components for Berkovich dynamical systems, due to Rivera-Letelier \cite{Rivera-Letelier_Asterisque_2003}.  The original proof is for rational functions over $\CC_p$; Jan Kiwi has written a proof for tame rational functions that we include as Appendix~\ref{Sec: Appendix}.  The proof also uses a ``No Wandering Domains'' result of Benedetto \cite[Thm.~5.1]{Benedetto_wandering_nontriv_reduction}, from which we deduce that all Type II points in the Julia set are preperiodic (under the hypotheses of Theorem~D); see Proposition~\ref{Prop: Preperiodic Julia points}. The general strategy of the proof should carry over to $p$-adic fields, but the corresponding result on wandering domains is not known in full generality in residue characteristic~$p$ \cite{Benedetto_wild_recurrent_critical}. 

	
	As an additional application of the previous theorems, we look at the question of when a meromorphic map $F: X \dashrightarrow X$ of a complex surface can be resolved to be algebraically stable by a proper modification of $X$. Recall that $F$ is \textbf{algebraically stable} if there does not exist a curve $C$ such that $F^n(C)$ is collapsed to an indeterminacy point of $F$ for some $n \geq 1$.  The maps we consider preserve a rational fibration, and so locally take the form 
		$$(t,x) \mapsto (t, f_t(x)),$$
where $f_t$ is a meromorphic family of (complex) rational functions, with $t$ in the unit disk $\D$.  Our Theorem~D implies the following statement, in which $\hat \CC$ denotes the Riemann sphere. 

\begin{thmE}
	Let $f_t: \hat \CC \to \hat \CC$ be a meromorphic family of rational functions for $t \in \DD$ with $\deg(f_t) = d \geq 2$ for $t \neq 0$. Let $\pi: X \to \DD$ be a normal connected surface with projective fibers such that $\pi^{-1}(\DD^*) \cong \DD^* \times \hat \CC$ and such that each irreducible component of $X_0 = \pi^{-1}\{0\}$ has multiplicity~1. Consider the rational map $F: X \dashrightarrow X$ defined by $(t,z) \mapsto (t,f_t(z))$. There exists a proper modification $Y \to X$ that restricts to an isomorphism over $\DD^*$ such that the induced rational map $F_Y : Y \dashrightarrow Y$ is algebraically stable. 
\end{thmE}

In the context of complex projective surfaces, algebraic stability of $F$ is equivalent to the condition that $(F^*)^n = (F^n)^*$ for all $n \geq 1$ as operators on the Picard group.  If $F$ is bimeromorphic, then Diller and Favre have shown that an algebraically stable modification always exists \cite{Diller-Favre_Bimeromorphic_2001}.  Favre provided examples of monomial maps that show that this is not necessarily the case when $F$ is not birational \cite{Favre_Monomial_Examples}.  However, Favre and Jonsson have shown that for polynomial maps $F: \C^2\to\C^2$, a projective compactification $X$ of $\C^2$ always exists for which some iterate of $F: X \dashrightarrow X$ is algebraically stable \cite{FJ_Dynamical_Compactification}. We note that the latter article uses dynamics on a valuation space similar to the Berkovich line in order to deduce the existence of an algebraically stable resolution; the dynamics, and consequently their arguments, are of a very different flavor than those used here.  The new article of Diller and Lin also addresses the existence of algebraically stable modifications for rational maps on surfaces \cite{Diller-Lin}.  

\begin{remark}
The hypothesis on the surface $X$ in Theorem~E means that we only consider surfaces obtained from the product surface $\D\times\P^1$ by an inductive sequence of blow-ups over the smooth points of the central fiber $\{t=0\}$, followed by the blow-down of any subset of these exceptional curves.  In addition, Theorem~E asserts no control over the singularities of the proper modification~$Y$. We expect that it is possible to find $Y$ with at worst quotient singularities over $t = 0$, but we were unable to deduce it from our methods. 
\end{remark}

Finally, we observe that there is a connection between our notion of analytic stability and the arithmetic-dynamical notion of weak N\'eron model \cite{Hsia_Weak_Neron_1996,Briend-Hsia_wNm}. Let $\ell$ be a discretely valued field with valuation ring $\OO$, and let $f \in \ell(z)$ be a rational function. A weak N\'eron model for $(\PP^1_{\ell}, f)$ is a pair $(\XX, F)$ consisting of a regular semistable $\OO$-scheme $\XX$  and a rational map $F: \XX \to \XX$ such that:
	\begin{itemize}
		\item The generic fiber of $\XX$ is $\PP^1_{\ell}$;
		\item $F$ restricts to $f$ on the generic fiber; and
		\item $F$ restricts to a morphism on the smooth locus $\XX^{\mathrm{sm}}$.
	\end{itemize}
(This definition is slightly stronger than the one in \textit{loc. cit.}, but the proofs implicitly assume they are equivalent.)  Using the arguments in the present paper, one can show the following: Given a weak N\'eron model for $(\PP^1_{\ell}, f)$, one may associate to it a canonical vertex set $\Gamma$ such that $(f, \Gamma)$ is analytically stable, and the associated probability vector of Theorem~C has finite support. The converse (that if $(f,\Gamma)$ is analytically stable, one can associate to it a weak N\'eron model) is not true, though, as illustrated by Examples~\ref{Ex: Needs resolution} and~\ref{Ex: Different coordinates}. Both examples deal with a quadratic polynomial $f$ that admits repelling fixed points defined over the field $\ell = \CC(\!(\sqrt{t})\!)$, which preclude the existence of a weak N\'eron model for $(\PP^1_\ell, f)$  \cite{Hsia_Weak_Neron_1996}. Benedetto and Hsia have independently found the connection between weak N\'eron models and certain special vertex sets \cite{Benedetto-Hsia_wNm}.

\medskip

\noindent \textbf{Acknowledgments.} We would like to thank Rob Benedetto, Jeff Diller, Charles Favre, Liang-Chung Hsia, and Mattias Jonsson for valuable conversations. Finally, we are indebted to the anonymous referees for discovering small errors in earlier versions of our manuscript.  This research was supported by the US National Science Foundation DMS-1302929 and DMS-1517080; Jan Kiwi was supported by the Chile Fondecyt 1110448.

\bigskip
\section{Markov chains and equilibrium measures}
\label{Sec: Markov}

	Our goal in this section is to prepare and prove Theorem~C. In \S\ref{Sec: Counting} we provide a count of preimages in simple domains in terms of locally defined multiplicities. We define a discrete counterpart to the Fatou and Julia sets in \S\ref{Sec: vertex sets}, and in \S\ref{Sec: vertex stability} we define the notion of analytic stability and explore some of its properties. The final subsection contains the proof of Theorem~C.

\begin{convention}
	Throughout this section, we let $k$ be an algebraically closed field that is complete with respect to a nontrivial non-Archimedean absolute value. Note that we do not assume anything about the characteristic or residue characteristic of $k$. The Berkovich projective line over $k$ will be denoted $\Berk$ for brevity. An open disk $D$ in $\Berk$ is an open set with a unique boundary point, say $x$. If $\vec{v} \in T\Berk_x$ is the inward tangent vector defining $D$, we write $D = D(\vec{v})$. For a nonconstant rational function $f \in k(z)$, we write $m_f$ and $s_f$ for the local degree and surplus multiplicity, respectively. (See, e.g., \cite[\S3]{Faber_Berk_RamI_2013}.)  For an open disk $D = D(\vec{v})$ in $\Berk$ with boundary point $x$, we set 
		$$\bar f(D) := D(Tf(\vec{v})),$$
where $Tf: T\Berk_x \to T\Berk_{f(x)}$ is the action on the tangent space.  Note that $\bar f (D) = f(D)$ if and only if the surplus multiplicity $s_f(D)$ is zero.
\end{convention}


\subsection{Counting preimages}
\label{Sec: Counting}
	
	A \textbf{simple domain} $V \subset \Berk$ is an open set with finitely many boundary points, all of which must be of type~II or type~III. Equivalently, a simple domain is the intersection of finitely many open disks $V_i$, where $\partial V_i = \{x_i\}$ is a type~II or~III point for each index $i$. If each $V_i$ has inward tangent vector $\vec{v}_i$ at $x_i$, then we can also write $V = \cap D(\vec{v}_i)$. 

\begin{prop}
\label{Prop: pre-image decomposition} 
	Let $f: \Berk \to \Berk$ be a rational function of degree $d \geq 1$, and let $V$ be a simple domain written as the intersection of $n$ open disks $V_1, \ldots, V_n$.  
\begin{enumerate}
\item For any $y\in\Berk$, 
\[
\#\left(f^{-1}(y) \cap V\right) = \sum_{i \, : \,y\in \bar f(V_i)} m_f(V_i) \; + \, \sum_{i=1}^n s_f(V_i) \, -\, d(n-1),
\]
where all preimages are counted with multiplicity.
\item The image of $V$ under $f$ determines a partition of $\Berk$ into sets
\[
V_{f,I} := \bigcap_{i \in I} \bar f(V_i) \smallsetminus \bigcup_{i \not\in I} \bar f(V_i) \qquad \text{for 
$I \subset \{1, \ldots, n\}$}.
\]
The function $y \mapsto \#\left(f^{-1}(y) \cap V\right)$ is constant on each $V_{f,I}$. 
\end{enumerate}
\end{prop}


\begin{proof}
For each $y \in \Berk$, let $I(y)$ denote the (possibly empty) set of indices in $\{1, \ldots, n\}$ such that $i \in I(y)$ if and only if $y \in \bar f(V_i)$. 
Properties of the local degree and surplus multiplicity of an open disk \cite[Prop.~3.10]{Faber_Berk_RamI_2013} imply that
	$$\#\left(f^{-1}(y) \cap V_i \right) = \varepsilon(i,I(y)) \cdot m_f(V_i) + s_f(V_i)$$
for each $i = 1, \ldots, n$. 
Here $\varepsilon(\cdot,I)$ denotes the indicator function for a set $I$ of indices.

Choose an index $j\in\{1, \ldots, n\}$. The set of pre-images $f^{-1}(y) \cap V$ is precisely the set of pre-images in $V_j$ less the pre-images in the complement of $V_i$ for each index $i \neq j$. Therefore, 
	\begin{align*}
		\#\left(f^{-1}(y) \cap V\right) &= \#\left(f^{-1}(y) \cap V_j\right) 
			- \sum_{i\neq j} \#\left(f^{-1}(y) \cap V_i^c \right) \\
			&= \varepsilon(j,I(y)) \cdot m_f(V_j) + s_f(V_j) 
			- \sum_{i\neq j} 
			\Big(d - \left[\varepsilon(i,I(y)) \cdot m_f(V_i) + s_f(V_i) \right] \Big) \\
			&= \sum_{i \in I(y)} m_f(V_i) + \sum_{i=1}^n s_f(V_i)  - d(n-1).   
	\end{align*}
This completes the proof of the first assertion. The second follows immediately from the observation that, for any subset $I \subset \{1, \ldots, n\}$,  $V_{f,I} = \{y \in \Berk \ : \ I(y) = I \}$. 
\end{proof}


\subsection{Vertex sets}
\label{Sec: vertex sets}

	We now define a discretized counterpart to the Fatou and Julia sets. It is not an exact analogue, though the canonical measure for $f$ is always supported inside the discrete counterpart of the Julia set. 

\begin{define}
	A \textbf{vertex set} for $\Berk$ is a finite nonempty set of type~II points, which we typically denote by $\Gamma$.  The connected components of $\Berk \smallsetminus \Gamma$ will be referred to as \textbf{$\Gamma$-domains}. As a special case, when a $\Gamma$-domain has one boundary point, we call it a \textbf{$\Gamma$-disk}. Write $\SS(\Gamma)$ for the partition of $\Berk$ consisting of the elements of $\Gamma$ and all of its $\Gamma$-domains. 
\end{define}

\begin{define}
	Let $\Gamma$ be a vertex set for  $\Berk$, and let $f: \Berk \to \Berk$ be a rational function with $\deg(f) \geq 2$. A $\Gamma$-domain $U$ will be called an \textbf{$F$-domain} if $f^n(U) \cap \Gamma$ is empty for all $n \geq 1$, and otherwise $U$ will be called a \textbf{$J$-domain}. If $U$ is a $\Gamma$-disk, it will be called an \textbf{$F$-disk} or a \textbf{$J$-disk}, respectively. Write $\JJ(\Gamma) \subset \SS(\Gamma)$ for the subset consisting of all $J$-domains and the elements of $\Gamma$. 
\end{define}

	While the partition $\SS(\Gamma)$ is typically uncountable, the set $\JJ(\Gamma)$ is much smaller.

\begin{lem}
\label{Lem: J countable}
	For a given vertex set $\Gamma$, the set $\JJ(\Gamma)$ is countable. 
\end{lem}

\begin{proof}
	Since $\Gamma$ is finite, it suffices to show that the set of $J$-domains is countable. For each integer $n \geq 1$, write
		\[
			\JJ(\Gamma)_n  = \{U \text{ a $J$-domain} \ : \ f^n(U) \cap \Gamma \neq \varnothing, 
				f^i(U) \cap \Gamma = \varnothing \text{ for } i = 1, \ldots, n-1\}.
		\]
By definition, each $J$-domain lies in some $\JJ(\Gamma)_n$. Since $\Gamma$ is finite, and since each vertex has at most $d^n$ distinct pre-images under $f^n$, it follows that $\JJ(\Gamma)_n$ is finite for each $n \geq 1$.  
\end{proof}

\begin{prop}
\label{Prop: Fatou F-ensemble}
The Julia set for $f$ is contained in the union of the sets in $\JJ(\Gamma)$. 
\end{prop}

\begin{proof}
If $U$ is an $F$-domain, then $\Gamma$ does not intersect the union of the forward iterates $\bigcup_{n \geq 0} f^n(U)$, so that $U$ must lie in the Fatou set for $f$. Now the union of the $F$-domains is contained in the Fatou set for $f$, and taking complements shows that the Julia set is contained in the union of $\Gamma$ with all of the $J$-domains. 
\end{proof}


\subsection{Analytic stability}
\label{Sec: vertex stability}

	The following definition parallels the one used in the theory of rational self-maps of complex surfaces.  (We will discuss the correspondence in Section \ref{Sec: Resolutions of surfaces}.)

\begin{define}
\label{Def: Stable}
	Let $\Gamma$ be a vertex set for $\Berk$, and let $f: \Berk \to \Berk$ be a rational function with $\deg(f) \geq 2$. One says that the pair $(f, \Gamma)$ is \textbf{analytically stable} if for each $\zeta \in \Gamma$, either $f(\zeta) \in \Gamma$ or $f(\zeta) \in U$ for some $F$-domain $U$. 
\end{define}

\begin{lem}
\label{Lem: F-domains}
Suppose that $(f, \Gamma)$ is analytically stable.  For each $F$-domain $U$ there is another $F$-domain $V$ such that $f(U) \subset V$. 
\end{lem}

\begin{proof}
	Suppose that $f(U)$ is not contained in an $F$-domain. Then it is either contained in a $J$-domain $V$, or it is not fully contained in any $\Gamma$-domain. In the former case, no boundary point of $U$ can map into $V$ since $(f, \Gamma)$ is analytically stable. Hence $f(U) = V$, which forces $U$ to be a $J$-domain. In the latter case $f(U)$ must contain an element of $\Gamma$, which again implies that $U$ is a $J$-domain. In either case we have reached a contradiction. 
\end{proof}

\begin{lem}
\label{Lem: Multiplicity Definition}
Suppose that $(f, \Gamma)$ is analytically stable.  	Let $U, V \in \JJ(\Gamma)$ be $J$-domains or vertices.  Then the function 
		$$y \mapsto \#\left(f^{-1}(y) \cap V\right)$$ 
is constant on $U$, where pre-images are counted with multiplicities. 
\end{lem}


\begin{proof}
If $U$ is an element of $\Gamma$, then $U$ is a singleton and the result is evident.  If $U$ is a $J$-domain, then $f(\Gamma) \cap U = \varnothing$ by analytic stability, and  $\#\left(f^{-1}(y) \cap V\right) = 0$ whenever $V \in \Gamma$. 

Now suppose that $U$ and $V$ are $J$-domains. We may write $V$ as the intersection of $n$ open disks $V_1, \ldots, V_n$.  By Proposition~\ref{Prop: pre-image decomposition}, it suffices to show that $U$ lies in one of the subsets $V_{f,I}$ for an index set $I \subset \{1, \ldots, n\}$.  If no such index set exists, then $U$ must contain a boundary point of some $V_{f,J}$, say $y$, since $U$ is connected. But $y$ is the boundary point of one of the open disks $\bar f(V_j)$, so that the vertex $\zeta \in \partial V_j$ satisfies $y = f(\zeta) \in U$. That is, $f(\Gamma) \cap U \neq \varnothing$, a contradiction to analytic stability. 
\end{proof}

\begin{define}
\label{Def: Multiplicities}
Assume that $(f, \Gamma)$ is analytically stable.  	For each pair of elements $U,V \in \JJ(\Gamma)$, we define an integer $m_{U,V} \in \{0, \ldots, d\}$ as follows. For $y \in U$, set 
	 $$m_{U,V} = \#\left(f^{-1}(y) \cap V\right).$$ 
	By Lemma~\ref{Lem: Multiplicity Definition}, $m_{U,V}$ is independent of the choice of $y$. 
\end{define}

	The multiplicities $m_{U,V}$ are only well defined in the presence of analytic stability. Moreover, they are combinatorial in the sense that they can be computed via local mapping degrees and surplus multiplicities at points of $\Gamma$. More precisely, Proposition~\ref{Prop: pre-image decomposition} and the proof of Lemma~\ref{Lem: Multiplicity Definition} show that:
\[
	m_{U,V} = \begin{cases}
		m_f(V) & \text{for $U, V\in\Gamma$ with $f(V) = U$} \\
		\sum_{\zeta\in V: f(\zeta) = U} m_f(\zeta) & \text{for } U\in\Gamma \text{ and } J\text{-domain } V \\
		\sum_{i: U \subset \bar f(V_i)} m_f(V_i) + \sum_{i=1}^n s_f(V_i) - d(n-1) & \text{for } J\text{-domains $U$ and }V = \cap_{i=1}^n V_i  \\
		0 & \text{otherwise}
	\end{cases}
\]
The (non-negative integer valued) quantities $m_f(V)$, $m_f(V_i)$, and $s_f(V_i)$ may be determined algorithmically via reductions of $f$ in various coordinates; we stress that this is a finite computation.

\begin{lem}
\label{Lem: Consistent Multiplicities}
Suppose that $(f, \Gamma)$ is analytically stable.  	For each $U \in \JJ(\Gamma)$, we have
		\[
			\sum_{V \in \JJ(\Gamma)} m_{U,V} = \deg(f). 
		\]
\end{lem}

\begin{proof}
	Let $y \in U$ be any point. Observe that, by analytic stability and Lemma \ref{Lem: F-domains}, each pre-image of $y$ must be a vertex or else lie in a $J$-domain. 	The result now follows:
		\[
			\sum_{V \in \JJ(\Gamma)} m_{U,V} = \sum_{V \in \JJ(\Gamma)} 
			\#\left(f^{-1}(y) \cap V\right) = \#f^{-1}(y) = \deg(f). \qedhere
		\]		
\end{proof}


\subsection{The proof of Theorem~C }

	We restate Theorem~C now that we have all of the necessary definitions in hand. 
	
\begin{thmC}
	Let $k$ be an algebraically closed field that is complete with respect to a nontrivial non-Archimedean absolute value. Let $f : \Berk \to \Berk$ be a rational function defined over $k$ of degree $d \geq 2$, and let $\Gamma$ be a vertex set in $\Berk$. Suppose that $(f, \Gamma)$ is analytically stable and that the Julia set for $f$ is not contained in $\Gamma$. Writing $\JJ = \JJ(\Gamma)$ for the set of $J$-domains and vertices of $\Gamma$, we define a $|\JJ| \times |\JJ|$ matrix $P$ with $(U,V)$-entry
	\[
		P_{U,V} = \frac{m_{U,V}}{d}.
	\]
Then $P$ is the transition matrix for a countable state Markov chain with a unique stationary probability vector $\nu: \JJ \to [0,1]$. The rows of $P^n$ converge pointwise to $\nu$, and the $U$-entry of $\nu$ satisfies $\nu(U) = \mu_f(U)$ for each $U \in \JJ$, where $\mu_f$ is the equilibrium measure for~$f$. 	
\end{thmC}

	The following lemma makes the necessary connection between iterated pullback via $f$ and matrix multiplication.

\begin{lem}
\label{Lem: Matrix multiplication}
	Fix a $J$-domain or vertex $U_0 \in \JJ(\Gamma)$. For each $n \geq 1$, each $V \in \JJ(\Gamma)$, and each $y \in U_0$, we have
	\[
		\sum_{\substack{x \in V \\ f^n(x) = y}} m_{f^n}(x)
		= \sum_{U_1, \ldots, U_{n-1} \in \JJ(\Gamma)} m_{U_0,U_1} \cdot m_{U_1, U_2} \cdot m_{U_2, U_3}
				 \cdots m_{U_{n-1}, V}.
	\]
\end{lem}

\begin{proof}
	For ease of notation, let us write $\JJ = \JJ(\Gamma)$. The proof proceeds by induction. The case $n = 1$ follows immediately from the definitions:
	\[
		\sum_{\substack{x \in V \\ f(x) = y}} m_{f}(x) = \#\left(f^{-1}(y) \cap V\right) = m_{U_0,V}.
	\]
	
Now suppose that the result holds for some $n\geq 1$, and let us deduce it for $n+1$. We decompose the morphism $f^{n+1}: \Berk \to \Berk$ as $f^n \circ f$. Since the local degree $m_f(\cdot)$ is multiplicative, we see that
	\[
		\sum_{\substack{x \in V \\ f^{n+1}(x) = y}} m_{f^{n+1}}(x) 
		= \sum_{\substack{z \, :\, f^n(z) = y \\ z \in f(V)}} m_{f^n}(z) 
			\ \left( \sum_{\substack{x \, : \, f(x) = z \\ x\in V}} m_f(x) \right).
	\]
Let $z_n$ be a solution to the equation $f^n(z) = y$. As $y \in U_0$ and $U_0 \in \JJ(\Gamma)$, it follows that $z_n$ is either or a vertex or else lies in some $J$-domain, say $U_n$. Grouping the terms of the first sum above according to the distinct $J$-domains and vertices  $U_n \subset f(V)$, we find that
	\begin{align*}
		\sum_{\substack{x \in V \\ f^{n+1}(x) = y}} m_{f^{n+1}}(x) 
		&=  \sum_{\substack{U_n \in \JJ \\ U_n \subset f(V)}} 
			\sum_{\substack{z \, :\, f^n(z) = y \\ z \in U_n}} m_{f^n}(z) \ \cdot \ \#\left(f^{-1}(z) \cap V\right) \\
		&= \sum_{\substack{U_n \in \JJ \\ U_n \subset f(V)}} m_{U_n, V}
			\left(  \sum_{\substack{z \, :\, f^n(z) = y \\ z \in U_n}} 
			m_{f^n}(z) \right)\\
		&= \sum_{\substack{U_n \in \JJ \\ U_n \subset f(V)}} m_{U_n, V} 
			\sum_{U_1, \ldots, U_{n-1} \in \JJ} m_{U_0,U_1} \cdot m_{U_1, U_2} \cdot m_{U_2, U_3}
				 \cdots m_{U_{n-1}, U_n} \\
		&= \sum_{U_1, \ldots, U_n \in \JJ} 
			m_{U_0,U_1} \cdot m_{U_1, U_2} \cdot m_{U_2, U_3} \cdots m_{U_n, V}.
	\end{align*}
The second equality follows from the independence of $\#\left(f^{-1}(z) \cap V\right)$ in $z \in U_n$, while the second to last equality uses the induction hypothesis.	Note that $m_{U_n,V} = 0$ if $U_n$ is not contained in the image $f(V)$. 
\end{proof}

\begin{proof}[Proof of Theorem~C]
For ease of exposition, we present the proof in several steps. \\

\noindent \textbf{Step 1:} We claim that the quantities $P_{U,V}$ define a countable state Markov chain on the state space $\JJ$. It is clear the $0 \leq P_{U,V} \leq 1$ for all $U,V \in \JJ$, and we have already seen that $\JJ$ is countable (Lemma~\ref{Lem: J countable}). It remains to show that for each $U \in \JJ$, we have
		\[
			\sum_{V \in \JJ} P_{U,V} = 1.
		\]
This is precisely the content of Lemma~\ref{Lem: Consistent Multiplicities}. 

\medskip

\noindent \textbf{Step 2:} We claim that the vector $\omega_f: \JJ \to [0,1]$ given by $\omega_f(U) = \mu_f(U)$ is a stationary probability vector for $P$. We have already seen that each $F$-domain lies in the Fatou set of $f$ (Proposition~\ref{Prop: Fatou F-ensemble}), so that $\mu_f$ does not charge $F$-domains. Hence $\sum_{U \in \JJ} \omega_f(U) = 1$. The measure $\mu_f$ satisfies the pullback relation $f^* \mu_f = d \cdot \mu_f$ as Borel measures on $\Berk$ \cite{Favre_Rivera-Letelier_Ergodic_2010}. Integrating this formula against the characteristic function of $V$ yields
	\[
		\mu_f(V) = \frac{1}{d} \int_{\Berk} \#\left(f^{-1}(y) \cap V\right) \ \mu_f(y) 
			= \sum_{U \in \JJ}  \frac{m_{U,V}}{d} \ \mu_f(U) = \sum_{U \in \JJ} P_{U,V} \ \mu_f(U).
	\]
Evidently this is equivalent to $\omega_f = \omega_f P$. 

\medskip

\noindent \textbf{Step 3:} We claim that for each $n \geq 1$, each pair of subsets $U,V \in \JJ$, and each $y \in U$, we have
	\begin{equation}
	\label{Eq: Markov Measure formula}
		(P^n)_{U,V} = \int_V d^{-n} \left(f^*\right)^n \delta_y.
	\end{equation}
This statement is equivalent to the one given in Lemma~\ref{Lem: Matrix multiplication}. 

\medskip

\noindent \textbf{Step 4:} We claim that the invariant measure $\mu_f$ does not charge $\Gamma$. In general, $\mu_f$ can only charge a point of $\Berk$ if $f$ has simple reduction \cite[Cor.~10.47]{Baker-Rumely_BerkBook_2010}, in which case $\mu_f = \delta_{\zeta}$ for some type~II point $\zeta \in \Berk$. Here the Julia set is $\julia(f) = \{\zeta\}$. Our hypotheses ensure that $\zeta \not\in \Gamma$ if $f$ has simple reduction. So $\mu_f(\Gamma) = 0$.  

\medskip

\noindent \textbf{Step 5:} We claim that the matrix powers $P^n$ converge entry-by-entry to $\mathbf{1}\omega_f$, where $\mathbf{1}$ is the column vector of 1's and $\omega_f(U) = \mu_f(U)$ as in Step~4. Equivalently, $(P^n)_{U,V} \to \mu_f(V)$ as $n \to \infty$, where $(P^n)_{U,V}$ is the $(U,V)$-entry of $P^n$. 

	Fix a $J$-domain or vertex $U$ and choose any point $y \in U$ that is not a $k$-rational exceptional point for $f$. Favre and Rivera-Letelier's equidistribution of iterated pre-images \cite[Thm.~10.36]{Baker-Rumely_BerkBook_2010} shows that $d^{-n} \left(f^*\right)^n \delta_y \to \mu_f$ weakly as $n \to \infty$. (This is weak converge of Borel measures on $\Berk$.) Since $\mu_f$ does not charge the boundary of $V$ (Step~4), the result in Step~3 shows that
	\[
		(P^n)_{U,V} = \int_V d^{-n} \left(f^*\right)^n \delta_y \to  \mu_f(V). 
	\]

\medskip

\noindent \textbf{Step 6:} We claim that if $\nu$ is a probability vector such that $\nu P = \nu$, then $\nu = \omega_f$ (as in Step~2). By induction,  we have $\nu = \nu P^n$ for each $n \geq 1$. But then letting $n \to \infty$ gives
	\[
		\nu = \nu P^n \to \nu (\mathbf{1}\omega_f) = (\nu \mathbf{1}) \omega_f = \omega_f.
	\]
Note that passage to the limit involves interchanging the (possibly infinite) sum defining $\nu P^n$ and the limit of the sequence $P^n$. This is justified by dominated convergence and the fact that each entry of $P^n$ is bounded above by~1. Associativity is a similar consideration.  
\end{proof}

\bigskip
\section{Analytically stable augmentations of vertex sets}
\label{Sec: Stable augmentations}

	The goal of this section is to prove that (under suitable hypotheses) a vertex set in the Berkovich projective line can be enlarged to yield one that is analytically stable with respect to a given rational function. Some of the tools that we use require the residue characteristic of our field to be zero; however, much of our technique and the theory on which it is built carries over to dynamical systems over $\CC_p$ (see Remark~\ref{Rem: Tame extension}). 
		
\begin{convention}	
Throughout this section, we let $\ell$ be a discretely valued field with residue characteristic zero and completion $\hat \ell$, and we let $k$ be the completion of an algebraic closure of $\hat \ell$.  We write $\Berk = \Berk_k$. For $a \in k$ and $r \in \RR_{>0}$, we write $\disk(a,r)^-$ and $\disk(a,r)$ for the open and closed Berkovich disks centered at $a$ of radius $r$, respectively. 
\end{convention}

\begin{remark}
	For our application to rational maps on fibered complex surfaces, we will take $\ell$ to be the field of complex functions that are meromorphic at the origin in $\CC$. Then $\hat \ell = \CCt$ is the field of formal Laurent series in a local coordinate $t$, and $k = \LL$ is the completion of the field of formal Puiseux series. 
\end{remark}

A type~II point $\zeta$ is canonically associated with a norm on the function field $k(z)$ of $\PP^1_k$.  Following Berkovich, we write $|f(\zeta)|$ for the corresponding norm of  $f \in k(z)$; it is valued in $|k^\times|$. More concretely, there is a (classical) closed  disk $D(a,r) = \{x \in k : |x - a| \leq r\}$ with $a \in k$ and $r \in |k^\times|$ such that $|f(\zeta)| = \sup_{x \in D(a,r)} |f(x)|$ for any polynomial $f \in k[z]$. In this case, we write $\zeta = \zeta_{a,r}$. 

\begin{define}
\label{Def: ell-split}
	We say that a type~II point $\zeta \in \Berk$ is \textbf{$\ell$-split} if it is of the form $\zeta = \zeta_{a,r}$ for some $a \in \ell$ and $r \in |\ell^\times|$ --- i.e., if it is the supremum norm associated to an $\ell$-rational closed disk $\{x \in k \ : \ |x - a| \leq r\}$. A vertex set $\Gamma$ is \textbf{$\ell$-split}  if each of its elements is $\ell$-split. 
\end{define}

\begin{remark}
	By a density argument, a type~II point is $\ell$-split if and only if it is $\hat \ell$-split.
\end{remark}

The following criterion for a vertex to be $\ell$-split will be used in our application in \S\ref{Sec: Resolutions of surfaces}. 

\begin{prop}
\label{Prop: ell-split criterion}
Suppose that $\ell$ has algebraically closed residue field. Then a vertex $\zeta$ is $\ell$-split if and only if $|f(\zeta)| \in |\ell|$ for every rational function $f \in \ell(z)$. 
\end{prop}

\begin{proof}
Suppose that $\zeta$ is $\ell$-split, and write $\zeta = \zeta_{a,r}$ with $a \in \ell$ and $r \in |\ell^\times|$. By multiplicativity, it suffices to show $|f(\zeta_{a,r})| \in |\ell|$ for every polynomial $f \in \ell[z]$. Since $a \in \ell$, we can write $f = \sum c_n(z-a)^n$ for some coefficients $c_n \in \ell$. Then 
$|f(\zeta_{a,r})| = \max |c_n| r^n \in |\ell|$. 

Conversely, suppose that $|f(\zeta)| \in |\ell|$ for all $f \in \ell(z)$. Write $\zeta = \zeta_{a,r}$ for some $a \in k$ and $r \in |k^\times|$. We claim that the classical disk $D(a,r)$ contains an element of $\ell$. Let us assume otherwise, i.e., that $D(a,r) \cap \ell = \varnothing$, and derive a contradiction. 

For $b \in \ell$, set $f(z) = z - b = (a-b) + (z-a)$. Then $|f(\zeta_{a,r})| = \max(|a-b|, r) = |a-b|$. As we are assuming $|f(\zeta_{a,r})| \in |\ell|$, and since $b \not\in D(a,r)$, we conclude that 
\[
|a-b| \in |\ell^\times| \text{ for all $b \in \ell$}.
\]
Let $t$ be a uniformizer for $\ell$. Taking $b = 0$, we see that $|a| = |t^N|$ for some $N \in \ZZ$. 
We now claim that for any $n \geq N$, there is $b_n \in \ell$ such that $|a - b_n| < |t^n|$. 
We proceed by induction on $n$. Since the residue field of $\ell$ is algebraically closed, there is $c_N \in \ell$ 
such that $|a/t^N - c_N| < 1$. This is equivalent to $|a - c_N t^N| < |t^N|$; setting $b_N = c_N t^N$ 
starts the induction. Assume now that $|a - b_n| < |t^n|$ for some $b_n \in \ell$ and some $n \geq N$.
Since $|a - b_n| \in |\ell^\times|$, there is $m  > n$ such that $|a - b_n| = |t^m|$. Let $c_m \in \ell$ be 
such that $\left|\frac{a - b_n}{t^m} - c_m\right| < 1$. Setting $b_{n+1} = b_n + c_m t^m$, we find that 
$|a - b_{n+1}| < |t^m| \leq |t^{n+1}|$, which completes the induction step. Taking $n$ sufficiently large 
yields an element $b_n \in \ell$ such that $|a - b_n| < |t^n| < r$, which contradicts our assumption 
that $D(a,r) \cap \ell = \varnothing$. 

We have now shown that the classical disk $D(a,r)$ contains an element of $\ell$, so we may assume without loss that $a \in \ell$. Set $f(z) = z - a$ to get $|f(\zeta_{a,r})| = r \in |\ell|^\times$, which means $\zeta$ is $\ell$-split. 
\end{proof}

	We restate Theorem~D using the conventions and terminology of this section.

\begin{thmD}
	Let $f \in \ell(z)$ be a rational function of degree $d \geq 2$, and let $\Gamma$ be a vertex set in $\Berk$. Suppose that $\Gamma$ is $\ell$-split. Then there exists an $\ell$-split vertex set $\Gamma'$ containing $\Gamma$ such that $(f, \Gamma')$ is analytically stable. 
\end{thmD}

\begin{remark}
	For each vertex set $\Gamma \subset \Berk_k$, there is a finite (discretely valued) extension $\ell' / \ell$ such that 
	$\Gamma$ is $\ell'$-split. The equivalence of this statement of Theorem~D with the one in the introduction follows 	immediately from this observation.
\end{remark}

	The proof is given in \S\ref{Sec: Stable proof} after recalling several definitions and preliminary results. The overall strategy is to use Rivera-Letelier's classification of Fatou components  to organize the augmentation of $\Gamma$. If a vertex $\zeta$ lies in the Julia set for $f$, then $\zeta$ must be preperiodic, and we may append its forward orbit to $\Gamma$. If a vertex $\zeta$ lies in the Fatou set for $f$, then we augment $\Gamma$ with elements of the forward orbit of $\zeta$ along with several other carefully chosen points, depending on the type of Fatou component to which $\zeta$ belongs. We refer the reader to \cite[\S10]{Baker-Rumely_BerkBook_2010} for background on non-Archimedean Fatou/Julia theory. 
	
	
	
	
	


\subsection{Periodic Fatou components and Julia points}
\label{Sec: Fatou background}

	Versions of the next two results were proved by Rivera-Letelier for rational functions over $\CC_p$ \cite[\S4,5]{Rivera-Letelier_Asterisque_2003}. Some parts of the proofs carry over verbatim to the case of residue characteristic zero; additional work was done by Kiwi for the parts that do not (see Appendix~\ref{Sec: Appendix}). 

	A rational function $f \in k(z)$ acts on $\Berk$ by an open map. Therefore, the image of a Fatou component is again a Fatou component. Moreover, the Fatou set of $f$ agrees with the Fatou set of any iterate $f^n$. So to classify periodic Fatou components for $f$, it will suffice to study the fixed ones. 
	
	As in the complex setting, if $U$ is a Fatou component such that $f^n(U) = U$, and if $U$ contains an attracting periodic point $p$, then we say that $U$ is the \textbf{immediate basin of attraction} for $p$. Following Kiwi, we say that a periodic Fatou component of period~$n$ is a \textbf{Rivera domain} if $f^n$ induces a bijection of $U$ onto itself. 

\begin{prop}[Rivera-Letelier, Appendix~\ref{Sec: Appendix}]
\label{Prop: Classification}
	Let $k$ be an algebraically closed field of characteristic zero that is complete with respect to a nontrivial non-Archimedean absolute value. Let $f \in k(z)$ be a tame rational function\footnote{A rational function is \textbf{tame} is its Berkovich ramification locus is contained inside the convex hull of its critical points. Every rational function over a field with residue characteristic zero is tame.} of degree $d \geq 2$, and let $U$ be a fixed Fatou component for $f$.  Exactly one of the following holds:
\begin{enumerate}
\item $U$ is the immediate basin of attraction for a type~I fixed point, or
\item $U$ is a fixed Rivera domain whose boundary is a union of at most $d-1$ type~II periodic orbits. 
\end{enumerate}
\end{prop}

\begin{remark}
A fixed Rivera domain is a simple domain by the second part of the theorem. 
\end{remark}

	The next proposition is the ``No Wandering Domains'' result that was alluded to in the Introduction. Note that $f$ must be defined over the discretely valued subfield $\ell \subset k$. 
	

\begin{prop}[Benedetto, {\cite[Thm.~5.1]{Benedetto_wandering_nontriv_reduction}}]
\label{Prop: Benedetto Wandering}
Let $\ell$ be a discretely valued field and let $f \in \ell(z)$ be a rational function of degree $d \geq 2$. If $U$ is a wandering Fatou component, then $f^n(U)$ is a disk with periodic type~II boundary point for all $n \gg 0$. Moreover, if $U$ contains an element of $\PP^1(\ell)$, then the boundary point of $f^n(U)$ is $\ell$-split. 
\end{prop}


	The following result is a very powerful consequence of the hypothesis that $f$ is defined over a discretely valued subfield. (Compare Example~\ref{Ex: Tent map}.)
	
\begin{prop}
\label{Prop: Preperiodic Julia points}
	Let $f \in \ell(z)$ be a rational function of degree $d \geq 2$. Then every type~II point of the Julia set of $f$ is preperiodic. 
\end{prop}

\begin{proof} 
	Without loss, we may assume that $\ell$ is complete. 
	Consider a complete extension $\ell'/ \ell$ with transcendental residue extension and trivial value group extension. (For example, let $\| \cdot \|$ be the Gauss norm on the Tate algebra $\ell\langle\tau \rangle$ in the variable $\tau$; write $\ell'$ for its fraction field, which is a complete non-Archimedean field with the same value group as $\ell$ and residue field $\tilde \ell' = \tilde \ell (\tau)$, a rational function field.) 
	
	Let $k'$ be a minimal complete and algebraically closed non-Archimedean field containing both $k$ and $\ell'$. It is isomorphic to the completion of an algebraic closure of $\ell'$. Let $\iota: \Berk_k \hookrightarrow \Berk_{k'}$ be the natural continuous embedding. (See \cite[\S4]{Faber_Berk_RamI_2013}.) Writing $f_{k'}$ for the induced map on $\Berk_{k'}$, we have that $\iota \circ f = f_{k'} \circ \iota$. In particular, $\iota$ induces an identification of Julia sets $\julia(f_{k'}) = \iota\left(\julia(f)\right)$. For ease of notation, we drop the use of the embedding map $\iota$ and view $\Berk_k$ as a closed subset of $\Berk_{k'}$. 
	
	Let $\zeta \in \julia(f)$ be a type~II point. Since $\tilde k \subsetneq \tilde k'$, there is a tangent direction $\vec{v} \in T\Berk_{k', \zeta} \smallsetminus T\Berk_{k, \zeta}$ such that the associated open disk $D(\vec{v})$ does not meet the Julia set. That is, $D(\vec{v})$ is a Fatou component. If it is a preperiodic component, then evidently its boundary point $\zeta$ is also preperiodic. Otherwise, $D(\vec{v})$ is a wandering component, and Proposition~\ref{Prop: Benedetto Wandering} applies since $\ell'$ is discretely valued. 
\end{proof}

	The next result will allow us to control the orbits of vertices in Rivera domains. 
	
\begin{lem}
\label{Lem: Rivera periodic locus}
	Let $f \in k(z)$ be a rational function of degree $d \geq 2$, and let $U$ be a Rivera domain for $f$. Then the set of periodic points in $U$ is closed and connected (in $U$).
\end{lem}

\begin{proof}
	Replacing $f$ with an iterate if necessary, it suffices to assume that $U$ is fixed. Write $\Sigma(\bar U)$ for the skeleton of $\bar U$ --- i.e., the connected hull of the boundary of the closure of $U$. As $\partial U$ is a finite nonempty set and $f$ acts on $U$ by an automorphism, the skeleton $\Sigma(\bar U)$ is fixed pointwise by some iterate of $f$. Without loss, we may assume that $\Sigma(\bar U)$ is itself fixed pointwise. Note that the connected components of $U \smallsetminus \Sigma(\bar U)$ are open disks. To prove the lemma, it suffices to show that for each such disk $D$, the periodic locus in $\bar D$ is closed and connected. For then the complement of the periodic locus in $U$ is a collection of disjoint open disks. 
	
	Suppose that $D$ is a connected component of $U \smallsetminus \Sigma(\bar U)$. Note that the boundary point of $D$ is fixed. If $D$ is not periodic, then the periodic locus in $\bar D$ is simply its boundary point. If $D$ is periodic, then we may assume without loss that it is fixed by $f$. Let $\eta$ be the boundary point of $D$. Observe that if $x \in D$ is any periodic point, say with period $n$, then the entire segment $[x,\eta]$ must be fixed by $f^n$. For $f^n$ is unramified along $[x,\eta)$; hence, $f^n([x,\eta])$ and $[x,\eta]$ have the same length and the same endpoints, and so they must agree pointwise. It follows that the set of periodic points in $D$ is connected. 
		
	To show that the periodic locus in $\bar D$ is closed, it suffices to show that there are only finitely many periods that can occur for a periodic point in $D$. Indeed, the set of points in $\bar D$ with period dividing a given integer~$n$ is closed, being the solutions to the equation $f^n(z) = z$. In fact, we will show something stronger: the set of periods that can occur for a periodic point in $\bar D$ contains at most two elements. Make a change of coordinate so that $D = \disk(0,1)^-$, in which case $\eta$ is the Gauss point. Since $f$ is an automorphism of $\disk(0,1)^-$, we may expand $f$ as 
	\begin{equation}
	\label{Eq: Form of f}
		f(z) = a_0 + a_1z + a_2 z^2 + \cdots, 
	\end{equation}
where $|a_i| \leq 1$ for all $i$, $|a_0| < 1$, and $|a_1| = 1$. Let $\lambda$ be the image of $a_1$ in the residue field $\tilde k$. If $\lambda$ is not a root of unity, we will show that $1$ is the only possible period for a periodic point in $\bar D$. If $\lambda$ is an $e^{\mathrm{th}}$ root of unity, then $\{1, e\}$ are possible periods. 
	
	Let $y \in D$ be a periodic point of period $n \geq 1$. Let $x_1 \in [y, \eta]$ be the closest fixed point to $y$. Note that $x_1$ is of type~II and $x_1 \neq \eta$ because the tangent vector $\vec{0} \in T\Berk_{\eta}$ pointing toward~$0$ is fixed by $f$. If $y = x_1$, then $n = 1$, and we are finished. Otherwise, the direction $\vec{v} \in T\Berk_{x_1}$ containing $y$ is periodic and non-fixed. 
	
	Writing $x_1 = \zeta_{b,|\rho|}$ for some $b, \rho \in k \cap \disk(0,1)^-$, we make a change of coordinate in \eqref{Eq: Form of f} to obtain the action of $f$ on $T\Berk_{x_1}$:
	\[
		\rho^{-1}f(b+ \rho z) - b \rho^{-1} = \frac{f(b) - b}{\rho} + a_1 z + \varepsilon(z),
	\]
where $\varepsilon(z)$ is a series whose coefficients all have absolute value strictly smaller than~1. Since $x_1$ is fixed, we find that $|f(b) - b| \leq \rho$. Let $\beta$ be the image of $\rho^{-1}(f(b) - b)$ in the residue field of $k$. Reducing the above expression modulo the maximal ideal of $k^{\circ}$ shows that, in appropriate coordinates, the action of $f$ on $T\Berk_{x_1}$ is given by 
	\[
		z \mapsto \beta + \lambda z.
	\]	 
Since there exists a periodic non-fixed direction at $x_1$, and since $\mathrm{char}(\tilde k) = 0$, $\lambda$ must be a nontrivial root of unity. Let $e > 1$ be the multiplicative order of $\lambda$. 

	We claim that $y$ has period~$e$. For if not, then let $x_e$ be the point closest to $y$ that is fixed by $f^e$. Then $x_e$ is of type~II, and $x_e \neq x_1$ since $f^e$ fixes the direction $\vec{v} \in T\Berk_{x_1}$ pointing toward $y$. We know $y$ lies in a periodic non-fixed direction at $x_e$. But $f^e$ acts on $T\Berk_{x_e}$ by $z \mapsto z + \beta'$ for some $\beta' \in \tilde k$ by a computation analogous to the one in the previous paragraph. This tangent map has no periodic non-fixed direction since $\mathrm{char}(\tilde k) = 0$. Hence $y  = x_e$ and $y$ has period~$e$. 
\end{proof}


\subsection{Existence of analytically stable augmentations}
\label{Sec: Stable proof}

	The goal of this section is to prove the following more precise version of Theorem~D. Recall that any $F$-domain $U$ has the property that $f(U)$ is contained in some $F$-domain (Lemma~\ref{Lem: F-domains}). The \textbf{forward orbit} of $U$ is the set of $F$-domains $V$ such that $f^n(U) \subset V$ for some $n \geq 0$.  We say that an $F$-domain  is \textbf{wandering} if it has infinite forward orbit.

\begin{thm}  \label{refined Theorem D}
	Let $f \in \ell(z)$ be a rational function of degree $d \geq 2$, and let $\Gamma$ be an $\ell$-split vertex set in $\Berk$. 
	There exists an $\ell$-split vertex set $\Gamma' \supset \Gamma$ such that for each $\zeta \in \Gamma'$, exactly one of the following is true: 
		\begin{itemize}
			\item $f(\zeta) \in \Gamma'$; 
			\item $f(\zeta)$ lies in a wandering $F$-disk (relative to $\Gamma'$) with periodic boundary point in $\Gamma'$; or
			\item $f(\zeta)$ lies in an $F$-disk (relative to $\Gamma'$) that contains an attracting type~I periodic point. 			
		\end{itemize}			
	In particular, $(f, \Gamma')$ is analytically stable. 
\end{thm}

	Setting $\Gamma_0 = \Gamma$, we will successively construct vertex sets $\Gamma_n \supset \Gamma_{n-1}$ for which fewer of the vertices in $\Gamma_n$ fail to meet one of the properties in the statement of the theorem. Or rather, for ease of notation, we will speak of ``enlarging the vertex set $\Gamma$'' at every step in this inductive procedure and dispense with the subscripts entirely. 

	Before beginning the proof in earnest, we note that a vertex $\zeta \in \Gamma$ is either in the Julia set $\julia(f)$, or else it lies in a Fatou component $U$. Then either $U$ is a wandering component, or else there exists $m \geq 0$ such that $f^m(U)$ is a Rivera domain or the basin of attraction of a periodic point (Proposition~\ref{Prop: Classification}). We will treat each of these cases separately, and then conclude by proving that what we have accomplished is sufficient for the theorem. 

\medskip

\noindent \textbf{Step 1:} Julia vertices.  Each element of $\Gamma \cap \julia(f)$ is preperiodic (Proposition~\ref{Prop: Preperiodic Julia points}). Enlarge $\Gamma$ by adjoining the forward orbits of all such points; now $\zeta \in \Gamma \cap \julia(f)$ implies $f(\zeta) \in \Gamma$. 

\medskip

\noindent \textbf{Step 2:} Wandering Fatou components.	Suppose that $\Gamma$ has nonempty intersection with a wandering Fatou component, say $U$. Let $\zeta_1, \ldots, \zeta_s$ be the vertices in the grand orbit of $U$. Then there exist integers $n_1, \ldots, n_s \geq 0$ such that $f^{n_1}(\zeta_1), \ldots, f^{n_s}(\zeta_s)$ lie in the same Fatou component, say $V$, and $f^m(V) \cap \Gamma = \varnothing$ for all $m \geq 1$. We may further assume that $V$ is a disk whose boundary point is periodic and $\ell$-split (Proposition~\ref{Prop: Benedetto Wandering}); let $\OO_V$ be the orbit of the boundary of $V$. Adjoin
	\[
		\OO_V \quad \text{and} \quad \bigcup_{j = 1}^s \Big\{ f(\zeta_j), f^2(\zeta_j), \ldots, f^{n_j}(\zeta_j) \Big\}
	\]
to $\Gamma$. Then $\Gamma$ is still $\ell$-split and finite. By construction, $f^m(V)$ is an $F$-disk for all $m \geq 1$. If $\zeta$ is a vertex in the grand orbit of $U$ such that $f(\zeta) \not\in \Gamma$, then our construction shows $f(\zeta) \in f(V)$, a wandering $F$-disk. 
	
\medskip

\noindent \textbf{Step 3:} Attracting basins. Suppose that $\Gamma$ has nonempty intersection with a preperiodic component $U$ such that $f^m(U)$ is the immediate basin of attraction of a (type~I) periodic point~$x$. For ease of exposition, we will explain the case where $x$ is a fixed point; the more general setting requires added notational effort only. 

We claim that $x \in \PP^1(\hat \ell)$, where $\hat \ell$ is the closure of $\ell$ in $k$. Indeed, since $\Gamma$ has nontrivial intersection with the basin of attraction of $x$, any small disk $D$ about $x$ will contain $f^n(\zeta)$ for some vertex $\zeta$ and all $n \gg 0$. In particular, there exists a sequence of elements $a_n \in \ell$ and radii $r_n \in |\ell^\times|$ such that $f^n(\zeta) = \zeta_{a_n, r_n}$, and such that the associated disk $\disk(a_n, r_n)$ is contained in $D$ for all $n \gg 0$. It follows that $a_n \to x$. 

Write $V = f^m(U)$ for the immediate basin of attraction of $x$, and let $D$ be a small closed disk about $x$ in $V$. We assume that $D$ is chosen small enough that $D \cap \Gamma$ is empty and that $f(D) \subsetneq D$. By the previous paragraph, we may further shrink $D$ if necessary in order to assume that its boundary point $\zeta_D$ is $\ell$-split. Let us adjoin $\zeta_D$ to $\Gamma$, so that $D \cap \Gamma = \{\zeta_D\}$. 

	By enlarging $\Gamma$ with sufficiently many iterates of its elements lying in the grand orbit of $V$, we may assume that (1) $D \cap \Gamma = \{\zeta_D\}$, and (2) if $\zeta \in \Gamma$ lies in the grand orbit of $V$, but $f(\zeta) \not\in \Gamma$, then $f(\zeta) \in D$. By (1) and the fact that $f(D) \subsetneq D$, each connected component of $D \smallsetminus \{\zeta_D\}$ is an $F$-disk.  

\medskip

\noindent \textbf{Step 4:} Rivera domains. Now suppose that $\Gamma$ has nonempty intersection with the grand orbit of a Rivera domain $U$. For ease of exposition, we will explain the case where $U$ is a fixed Rivera domain. By enlarging $\Gamma$ with sufficiently many iterates of its elements, we may assume that if $\zeta \in \Gamma$ lies in the grand orbit of $U$, but $f(\zeta) \not\in \Gamma$, then $\zeta$ lies in $U \smallsetminus P$, where $P$ is the periodic locus of $f$ in $U$. By Lemma~\ref{Lem: Rivera periodic locus}, $P$ is closed and connected in $U$. Thus $U \smallsetminus P$ is a disjoint union of open disks. 

	Let $D$ be a connected component of $U \smallsetminus P$. Observe that $f^m(D) \cap f^n(D) = \varnothing$ for all $m > n \geq 0$. Indeed, the boundary point of $D$ is periodic, but $D$ cannot itself be periodic. For otherwise it would contain periodic points close to its boundary, in contradiction to the fact that it is disjoint from the periodic locus. Moreover, if $D \cap \Gamma$ is nonempty, then $D$ contains an element $a \in \PP^1(\ell)$, and hence so do each of its iterates $f^n(D)$. Without loss, we may change coordinates in order to assume that the Rivera domain $U$ is contained in the unit disk $\disk(0,1)^-$. Now the diameter of $D$ is the difference $|a - f^n(a)| \in |\ell^\times|$, where $n$ is the period of the boundary of $D$. That is, the boundary point of $D$ is periodic and $\ell$-split. The argument given in Step 2 may now be applied to $D$ (in place of $V$) in order to show that if $\zeta \in \Gamma$ is a vertex in the grand orbit of $D$ such that $f(\zeta) \not\in \Gamma$, then $f(\zeta)$ lies in a wandering $F$-disk with periodic boundary point. 

\medskip

\noindent \textbf{Step 5:} Conclusion. After applying Step~1, we may assume that if $\zeta \in \Gamma$ and $f(\zeta)\not\in\Gamma$, then $\zeta$ lies in the Fatou set of $f$. There are finitely many grand orbits of Fatou components that meet $\Gamma$. After applying Step~2, we see that if $\zeta \in \Gamma$ and $f(\zeta)\not\in \Gamma$, then either $f(\zeta)$ lies in a wandering $F$-disk, or else $\zeta$ lies in the grand orbit of a periodic Fatou component. After applying Step~3, we conclude that if $\zeta \in \Gamma$ and $f(\zeta)\not\in \Gamma$, then $f(\zeta)$ lies in an $F$-disk containing an attracting periodic point, or it lies in a wandering $F$-disk with periodic boundary point, or else $\zeta$ lies in the grand orbit of a Rivera domain. Note that in Step~3, we have only added vertices in grand orbits of attracting basins, so this has no impact on the $F$-domains we created in Step~2. Finally, after applying Step~4 we have the conclusion of the theorem. Again, note that we have only added vertices in grand orbits of Rivera domains, which does not impact the work from Steps~2 or~3. This concludes the proof of Theorem~\ref{refined Theorem D} (and therefore also of Theorem~D).

\bigskip
\section{Algebraically stable resolutions of complex surfaces}
\label{Sec: Resolutions of surfaces}

	In this section, we prove Theorem~E. Throughout, we will write $\hat \CC$ for the Riemann sphere (in order to distinguish it from the complex scheme $\PP^1_\CC$). Recall the statement:

 \begin{thmE}
 	Let $f_t: \hat \CC \to \hat \CC$ be a meromorphic family of rational functions for $t \in \DD$ such that $\deg(f_t) = d \geq 2$ when $t \neq 0$. Let $\pi: X \to \DD$ be a normal connected proper fibered surface such that $\pi^{-1}(\DD^*) \cong \DD^* \times \hat \CC$ and each irreducible component of $X_0 = \pi^{-1}\{0\}$ has multiplicity~1. Consider the rational map $F: X \dashrightarrow X$ defined by $(t,z) \mapsto (t,f_t(z))$. There exists a proper modification $Y \to X$ that restricts to an isomorphism over $\DD^*$ such that the induced rational map $F_Y : Y \dashrightarrow Y$ is algebraically stable. 
 \end{thmE}
 
 \begin{remark}
 	The hypotheses of the theorem are trivially satisfied if one takes $X = \DD \times \hat \CC$.
 \end{remark}
	
	To apply our result on analytically stable augmentations of vertex sets in the Berkovich projective line, we must connect the dynamics of rational functions on $\Berk_{\LL}$ with the dynamics of rational maps on fibered surfaces. The connection between vertex sets and formal semistable fibrations of curves over complete valuation rings of height~1 is well understood \cite{BPR_section5,Bosch-Lutkebohmert_Stable_1}, and we will describe a modification of that theory that applies in our setting.  An alternate approach involving moving frames, as in Kiwi's article  \cite{Kiwi_Rescaling_Limits_2015}, may also be feasible.
	
	
\subsection{Models and vertex sets}

	We now recall and extend the discussion from \cite[\S5.1]{DeMarco_Faber_Degenerations_2014} on the relationship between (degenerating) families of rational curves over a complex disk and vertex sets in $\Berk_\LL$. As we will need to identify families over disks of varying sizes, it will be most convenient to work with schemes over the ring of holomorphic germs at the origin of a disk. 
	
	Write $\OO$ for the ring of germs of holomorphic functions at the origin in $\DD$; $\OO$ is a local PID with uniformizer $t$, say. Write $\ell$ for its fraction field. It is a discretely valued non-Archimedean field when endowed with the natural absolute value $\exp\left( - \ord_{t=0}(\cdot)\right)$ corresponding to vanishing at the origin. The completion of $\ell$ is the field of formal Laurent series $\CCt$. Write $\LL$ for the completion of an algebraic closure of $\CCt$, and write $\LL^\circ$ for its valuation ring. 

\begin{define}
	A \textbf{model of $\PP^1_\ell$} is a pair $(X, \eta)$ consisting of a normal connected projective $\OO$-scheme $X$ whose closed fiber is reduced at all codimension~1 points, and an isomorphism $\eta: X \times_{\Spec \OO} \Spec \ell \simarrow \PP^1_{\ell}$. We say that a model $(X, \eta)$ \textbf{dominates} a model $(X', \eta')$ if there is an $\OO$-morphism $g: X \to X'$ such that the map $ \eta' \circ g_{\ell} \circ \eta^{-1}$ is the identity on $\PP^1_{\ell}$. Two models will be called \textbf{isomorphic} if there is a dominating isomorphism between them. 
\end{define}


\begin{remark}
When discussing models, we will typically suppress mention of the map $\eta$ and simply write $X$. 
\end{remark}

\begin{remark}
\label{Rem: Scheme-to-complex}
In the language of complex geometry, a model $(X, \eta)$ may be interpreted as a normal connected projectively fibered surface $\pi: X \to \DD_\varepsilon$ for some small $\varepsilon > 0$, where $\DD_\varepsilon$ is the complex disk of radius $\varepsilon$. The isomorphism $\eta$ corresponds to a trivialization $\pi^{-1}(\DD_{\varepsilon}^*) \cong \DD_{\varepsilon}^* \times \hat \CC$. Note that shrinking $\varepsilon$ does not change the (isomorphism class of the) model, as we are only concerned with its germ structure. 
\end{remark}

	Let $X$ be a model of $\PP^1_\ell$. We claim that $X$ gives rise, canonically, to a vertex set $\Gamma_X \subset \Berk = \Berk_{\LL}$. The local ring of $\DD$ at the origin is contained inside $\LL^\circ$, and hence so is its completion. By completing along the central fiber $X_0$ and base extending to $\LL^\circ$, we obtain an admissible formal scheme $\XX$ over $\LL^\circ$ with generic fiber $\Berk$. Note that since $X_0$ is reduced in codimension~1, it may be identified with the special fiber $\XX_s$ as $\CC$-schemes.  Let 
		$$\red_X: \Berk \to X_0$$ 
be the canonical surjective reduction map \cite[2.4.4]{Berkovich_Spectral_Theory_1990}. Let $\eta_1, \ldots, \eta_r$ be the generic points of the irreducible components of the special fiber $X_0$. There exist unique type~II points $\zeta_1, \ldots, \zeta_r \in \Berk$ such that $\red_X(\zeta_i) = \eta_i$ for $i = 1, \ldots, r$. The desired vertex set is $\Gamma_X = \{\zeta_1, \ldots, \zeta_r\}$. 

\begin{lem}
For any model $X$ of $\PP^1_\ell$, the vertex set $\Gamma_X$ is $\ell$-split. 
\end{lem} 

\begin{proof}
Let $\eta$ be the generic point of an irreducible component $Z$ of the special fiber $X_0$. The corresponding point $\zeta \in \Gamma_X$ is constructed as follows. For $f \in \ell(z)$, we set $\|f\|_\eta := \exp(-\ord_Z(f))$. This is a norm on $\ell(z)$, and $\|c\|_\eta = 1$ for any $c \in \CC$. Since $X_0$ is reduced, we see that 
\[
\|t\|_\eta = \exp(-\ord_Z(t)) = e^{-1} = |t|,
\] 
so that this norm agrees with the given absolute value on $\ell$. By standard results in ultrametric analysis, $\|\cdot\|_\eta$  extends uniquely to a norm on $\LL(z)$ which extends the given absolute value on $\LL$. Now $\zeta$ is the point of $\Berk$ corresponding to the norm $\|\cdot \|_\eta$.

By construction, for $f \in \ell(z)$ we find that $|f(\zeta)| = \exp(-\ord_Z(f)) \in |\ell|$. An application of Proposition~\ref{Prop: ell-split criterion} completes the proof. 
\end{proof}

Recall from \S\ref{Sec: vertex sets} that we write $\SS(\Gamma_X)$ for the partition of $\Berk$ consiting of the elements of $\Gamma_X$ along with the connected components of $\Berk \smallsetminus \Gamma_X$ (i.e., the $\Gamma_X$-domains).

\begin{prop}
\label{Prop: Vertex sets and models}
	The association $X \mapsto \Gamma_X$ induces a bijection between the collection of isomorphism classes of models of $\PP^1_{\ell}$ and the collection of $\ell$-split vertex sets in $\Berk$. Moreover, the following are true:
	\begin{enumerate}
		\item Fix a model $X$. For each closed point $x \in X_0$, the formal fiber $\red_X^{-1}(x)$ is a $\Gamma_X$-domain. 
			The association $x \mapsto  \red_X^{-1}(x)$ induces a bijection between points of the $\CC$-scheme $X_0$ 
			and elements of $\SS(\Gamma_X)$.
			
		\item	If $X$ and $X'$ are models of $\PP^1$, then  $X$ dominates $X'$ if and only if $\Gamma_X \supset \Gamma_{X'}$. 
	\end{enumerate}
\end{prop} 

\begin{proof}[Sketch of proof]
	We have already shown that $X \mapsto \Gamma_X$ gives a well-defined $\ell$-split vertex set in $\Berk$. Functoriality of formal completion, the generic fiber construction, and the reduction map construction shows that if $X'$ and $X$ are isomorphic as models, then $\Gamma_{X'} = \Gamma_X$. Thus we have a well defined map between the desired collections of objects. The map is injective because models are determined by their formal fibers \cite[Lem.~3.10]{Bosch-Lutkebohmert_Stable_1}.
	
	We now sketch the proof that $X \mapsto \Gamma_X$ is surjective. Fix a vertex set $\Gamma$. The argument in \cite[Thm.~4.11]{BPR_section5} carries over to our setting mutatis mutandis and produces a formal model $\mathfrak X$ over $\hat \ell = \CCt$ with associated vertex set $\Gamma$. The vertex set $\Gamma$ allows us to define gluing data for $\Berk$ consisting of a finite union of closed affinoids with Shilov boundary in $\Gamma$. In order to pass to algebraic models, we need only observe that the canonical models of the closed affinoids in the proof of  \cite[Thm.~4.11]{BPR_section5} are formal completions of algebraic models over $\ell$. 
More precisely, if $T = \LL\langle z\rangle$ is the standard single variable Tate algebra, then the affinoid algebras in question are of the form 
	\[
		T \langle Y_1, \ldots, Y_m \rangle /  \left( (z-a_i)Y_i - c_i \; : \; i = 1, \ldots, m \right),
	\]
 for some $a_i, c_i \in \OO$ (by the $\ell$-split hypothesis). The associated $\OO$-algebra is given by 
		\[
			\OO[z, Y_1, \ldots, Y_m] / \left((z-a_i)Y_i - c_i \; : \; i = 1, \ldots, m \right).
		\]
As in the formal case, the associated local models over $\OO$ glue to give a global model $X$. A direct computation shows that $X$ is normal with reduced central fiber, and by construction we have $\Gamma_X = \Gamma$. 

	Now fix a model $X$. The formal fiber $\red_X^{-1}(x)$ is open for each closed point $x \in X_0$ by  anticontinuity of the reduction map. Since each one is disjoint from the vertex set $\Gamma_X$, it follows that $\red_X^{-1}(x)$ is a $\Gamma_X$-domain, and that $x \mapsto \red_X^{-1}(x)$ defines a bijection between the points of the $\CC$-scheme $X_0$ and the collection of $\Gamma_X$-domains and vertices $\SS(\Gamma_X)$. 
	
	The proof of the final claim about domination and vertex set containment follows exactly as in the formal case; see \cite[Thm.~4.11]{BPR_section5}. 
\end{proof}

\subsection{Proof of Theorem~E}

	Suppose that $f_t$ is a meromorphic family of rational functions with $\deg(f_t) = d \geq 2$ for $t \neq 0$, and that $X$ is a model of $\PP^1_\ell$. We identify $X$ with a normal fibered complex surface over a small disk as in Remark~\ref{Rem: Scheme-to-complex}. Define a rational map $F: X \dashrightarrow X$ by $F(t,z) = (t,f_t(z))$ for all $t \neq 0$. Let $f: \Berk \to \Berk$ be the rational function determined by viewing the parameter $t$ as an element of the field $\LL$. 

\begin{lem}
\label{Lem: Dynamic compatibility}
The action of $F$ on $X_0$ and the action of $f$ on $\SS(\Gamma_X)$ are compatible in the following sense: If $x,x' \in X_0$ are points (closed or generic) and $U_x = \red_X^{-1}(x)$ and $U_{x'} = \red_X^{-1}(x')$ are the 
			corresponding $\Gamma_X$-domains or vertices, then 
		\[
			F(x) = x' \quad \text{if and only if} \quad f(U_x) \subset U_{x'}. 
		\]
In particular, $x \in X_0$ is an indeterminacy point for $F$ if and only if $f(U_x)$ contains an element of the vertex set $\Gamma_X$.
\end{lem}

\begin{proof}
	If $F$ were a morphism, this would follow immediately from functoriality of reduction. To apply this argument, we begin by resolving the indeterminacy of $F$. Let $\rho: Y \to X$ be a model dominating $X$ such that the rational map $F$ extends to a morphism $\bar F: Y \to X$ satisfying $\bar F = F \circ \rho$ when the right side is defined. Functoriality of the reduction map gives a commutative diagram
	\[
		\xymatrix{
			\Berk \ar[r]^f \ar[d]_{\red_Y}& \Berk \ar[d]^{\red_X} \\ Y_0 \ar[r]^{\bar F_0} & X_0 
		}
	\]
Here we write $\bar F_0$ for the induced morphism on central fibers. For $y \in Y_0$, let us write $U_y = \red_Y^{-1}(y)$. It follows that 
	\begin{equation}
	\label{Eq: Correspondence on Y}
		\bar F_0(y) = x' \quad \text{if and only if} \quad f(U_y) \subset U_{x'}. 
	\end{equation}
Since domination of models corresponds to vertex set containment, we may partition the $\Gamma_X$-domain $U_X$ as
	\[
		U_x = \bigcup_{y \in \rho^{-1}(x)} U_y.
	\]
Applying \eqref{Eq: Correspondence on Y} simultaneously to all $y \in \rho^{-1}(x)$ gives the desired result. 
\end{proof}	
	
	We are now ready to prove Theorem~E. By a gluing construction, it suffices to produce a proper modification $Y \to X$ over a smaller disk; in particular, we may work with models over $\Spec \OO$ throughout. Identify $X$ with a model of $\PP^1_\ell$, and let $\Gamma_X \subset \Berk$ be the associated vertex set. There exists an $\ell$-split vertex set $\Gamma'$ containing $\Gamma_X$ such that the pair $(f, \Gamma')$ is analytically stable (Theorem~D). By Proposition~\ref{Prop: Vertex sets and models}, there is a model $Y$ of $\PP^1_\ell$ that dominates $X$ and has vertex set $\Gamma_Y = \Gamma'$. The rational map $F: X \dashrightarrow X$ extends to a rational map $F_Y: Y \dashrightarrow Y$, which we claim is algebraically stable. This is clear for horizontal curves (i.e., those that project onto $\Spec \OO$). So suppose that $C \subset Y_0$ is an irreducible curve such that $F_Y^n(C)$ is collapsed to an indeterminacy point for $F_Y$, say $y$. Let $\zeta \in \Gamma_Y$ be the vertex such that $\red_Y(\zeta)$ is the generic point of $C$, and let $U = \red_Y^{-1}(y) \in \SS(\Gamma_Y)$ be the $\Gamma_Y$-domain corresponding to $y$. By the preceding lemma, we see $f^n(\zeta) \in U$ and that $f(U)$ meets the vertex set $\Gamma_Y$. This means $U$ is a $J$-domain for $\Gamma_Y$. But the pair $(f, \Gamma_Y)$ is analytically stable, so we have a contradiction.

\bigskip
\section{Examples}
\label{Sec: Examples}
In this final section, we provide a collection of examples to illustrate Theorems C and D.  We also show that the results fail without hypotheses on the Julia set (in Theorem~C) and the field of definition (in Theorem~D).  We conclude with a comparison of analytic stability and the indeterminacy condition in the first author's earlier work \cite{DeMarco_Boundary_Maps_2005}.

To relate our discussion to complex dynamics and the previous article \cite{DeMarco_Faber_Degenerations_2014}, our examples are defined over $\LL$, the completion of the field of Puiseux series in the parameter $t$.   Note that $|t| = \exp(-1) < 1$.  The Gauss point of $\Berk$ (corresponding to the sup-norm on the unit disk) will be denoted by $\zeta_g$.  For $a \in \LL$ and $r \in \RR_{>0}$, we write $\disk(a,r)^-$ and $\disk(a,r)$ for the open and closed Berkovich disks centered at $a$ of radius $r$, respectively. 
	
\subsection{A straightforward computation with Theorem~C}
\label{Ex: Sample computation}
	Consider the pair $(f, \Gamma)$ given by 
		\[
			f(z) = z - 1 + t/z,  \qquad \qquad \Gamma = \{\zeta_g\},
		\]
with degree $d=2$.  Then $f(\zeta_g) = \zeta_g$, and the action on the tangent space $T\Berk_{\zeta_g} = \PP^1(\CC)$ is given by the translation $Tf(z) = z-1$.   The pair $(f, \Gamma)$ is analytically stable, and the Julia set is not contained in $\Gamma$. The second preimage of $\zeta_g$ lies in the disk $D = \disk(0,1)^-$; in fact, the disk $D$ has surplus multiplicity $s_f(D) = 1$.  The $J$-domains are disks of the form $U_a = \disk(a,1)^-$ for $a = 0, 1, 2,  \ldots$ The transition matrix $P$ for the associated Markov chain is given by 
	\[
		\bordermatrix{ 
			& \zeta_g & U_0 & U_1 & U_2 & U_3 & U_4 & \cdots \cr
			\zeta_g & 1/2 & 1/2 & 0 & 0 & 0 & 0 & \cdots \cr
			U_0 & 0 & 1/2 & 1/2 & 0 & 0 & 0 & \cdots \cr
			U_1 & 0 & 1/2 & 0 & 1/2 & 0 & 0 & \cdots \cr
			U_2 & 0 & 1/2 & 0 & 0 & 1/2 & 0 & \cdots \cr
			U_3 & 0 & 1/2 & 0 & 0 & 0 & 1/2 & \cdots \cr
			U_4 & 0 & 1/2 & 0 & 0 & 0 & 0 & \ddots    \cr						
			\vdots   & \vdots & \vdots & \vdots & \vdots & \vdots & \vdots & \ddots \cr
		}
	\]
Upon computing a few powers $P^n$, it is not difficult to see that the unique stationary probability vector in this case is
	\[
		\pi = \begin{pmatrix} 
			0 & 1/2  & 1/4 & 1/8 & 1/16 & \cdots 
			\end{pmatrix}.
	\]

\subsection{The failure of Theorem~C without a hypothesis on $\julia(f)$}
\label{Ex: Simple reduction}
	Theorem~C requires a hypothesis that the Julia set of $f$ is not contained in the vertex set $\Gamma$. To see it is necessary to make some assumption of this kind, consider the pair $(f, \Gamma)$ given by 
	\[
		f(z) = \frac{1}{z^2}, \qquad \qquad \Gamma = \{\zeta_g, \; \zeta_{0, |t|}, \; \zeta_{0, |t|^{-1}}\}.
	\]
Then $f$ has good reduction, so that $\julia(f) = \{\zeta_g\} \subset \Gamma$. Moreover, if $D_+$ (resp. $D_-$) is the open disk with boundary point $\zeta_{0, |t|^{+1}}$ (resp. $\zeta_{0, |t|^{-1}}$) and containing $0$ (resp. $\infty$), then $f(D_+) \subsetneq D_-$ and $f(D_-) \subsetneq D_+$. Hence $(f, \Gamma)$ is analytically stable. 

Let $A_{\pm}$ be the open annuli with boundary points $\zeta_g$ and $\zeta_{0, |t|^{\pm 1}}$, respectively. Then $A_{\pm}$ are the only $J$-domains, so that 
	\[
		\JJ(\Gamma) = \{A_+, A_-, \zeta_g, \zeta_{0, |t|}, \zeta_{0, |t|^{-1}}\}.
	\] 
As $f$ has local degree~2 along the segment $[0, \infty]$, we find that the transition matrix $P$ for our Markov chain is given by the matrix 
	\[
		\bordermatrix{ 
			& A_+ & A_- & \zeta_g & \zeta_{0, |t|} & \zeta_{0, |t|^{-1}} \cr
			A_+ & 0 & 1 & 0 & 0 & 0 \cr
			A_- & 1 & 0 & 0 & 0 & 0 \cr
			\zeta_g & 0 & 0 & 1 & 0 & 0  \cr
			\zeta_{0, |t|} &  0 & 1 & 0 & 0 & 0   \cr
			\zeta_{0, |t|^{-1}} & 1 & 0 & 0 & 0 & 0 \cr
		}		
	\]
Now $P^{2n} = P^2$ and $P^{2n+1} = P$ for $n \geq 1$. As $P^2 \neq P$, the powers of $P$ do not converge. Moreover, there are two independent stationary vectors for $P$, namely 
	\[
		\begin{pmatrix} 1/2 & 1/2 & 0 & 0 & 0\end{pmatrix} \qquad \text{and} \qquad \begin{pmatrix} 0 & 0 & 1 & 0 & 0 \end{pmatrix}.
	\]

\subsection{An analytically unstable example modified by the procedure of Theorem~D} 
\label{Ex: Needs resolution}
	Consider the pair $(f, \Gamma)$ with
		\[
			f(z) = z^2 + \frac{1}{t}, \qquad \qquad \Gamma = \{\zeta_g\}.
		\]
Set $U = \Berk \smallsetminus \disk(0,1)$. One computes that $f(\zeta_g) = \zeta_{1/t, 1} \in U$, and that $f(U) = \Berk \smallsetminus \disk(1/t, 1)$. Hence $\zeta_g \in f(U)$, and we conclude that $(f, \Gamma)$ is not analytically stable. 

	We observe that $\infty$ is an attracting fixed point for $f$, and $\zeta_{0,1/|t|}$ lies in its basin of attraction. Following the proof of Theorem~D, we define 
		\[
			\Gamma' = \Gamma \cup \{\zeta_{0, 1/|t|}\} = \{\zeta_g, \;\zeta_{0, 1/|t|}\}.
		\]
Set $D = \Berk \smallsetminus \disk(0, 1/|t|)^-$. Then $f(D) \subsetneq D$, so that the connected components of $D \smallsetminus \{\zeta_{0, 1/|t|}\}$ are $F$-disks. Since $f(\zeta_{0,1/|t|})$ lies in one of these disks, the pair $(f, \Gamma')$ is analytically stable. 

	Let $A$ be the open annulus with boundary points $\zeta_g$ and $\zeta_{0, 1/|t|}$. It is the unique $J$-domain for $\Gamma'$, so that 
		\[
			\JJ' = \{A, \; \zeta_g, \; \zeta_{0, 1/|t|}\}.
		\]
Note that $f^{-1}(\zeta_g) = \{\zeta_{1/\sqrt{-t}, \sqrt{|t|}}, \zeta_{-1/\sqrt{-t}, \sqrt{|t|}}\} \subset A$ and that $f^{-1}(\zeta_{0,1/|t|}) = \{\zeta_{0,1/\sqrt{|t|}}\} \subset A$. It follows that the transition matrix $P'$ is given by 
	\[
		\bordermatrix{ 
			& A & \zeta_g  & \zeta_{0, 1/|t|} \cr
			A & 1 & 0 & 0 \cr
			\zeta_g & 1 & 0 & 0 \cr
			\zeta_{0, 1/|t|} & 1 & 0 & 0  \cr
		}		
	\]
 Evidently $(P')^n = \mathbf{1}\nu$ for all $n \geq 1$, where $\nu = \begin{pmatrix} 1 & 0 & 0  \end{pmatrix}$, so that all of the mass of $\mu_f$ is contained in $A$. 

\begin{figure}[h]
\includegraphics[width=3.5in]{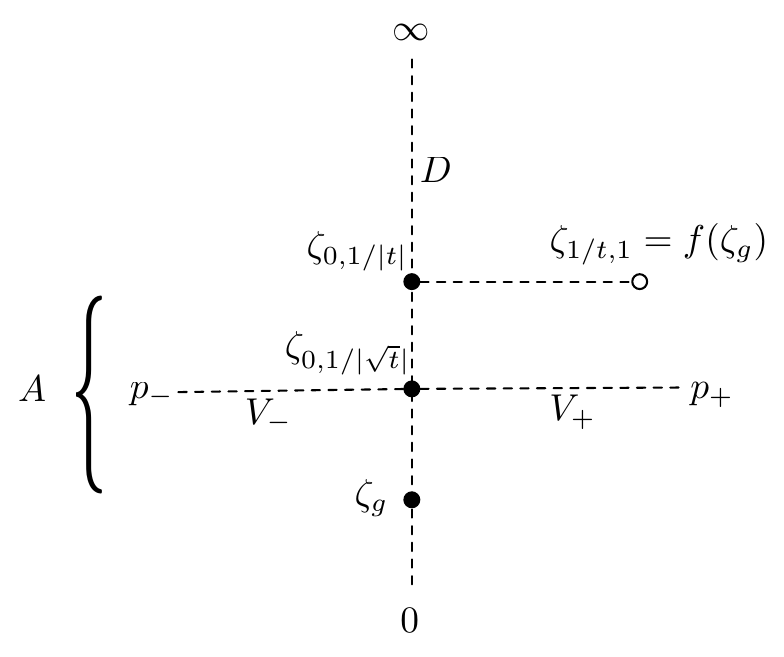}
\caption{A schematic representation of the state spaces $\JJ'$ and $\JJ''$ in Examples~\ref{Ex: Needs resolution} and~\ref{Ex: Different coordinates}. }
\label{quad}
\end{figure}

\subsection{The quadratic polynomial family, again}
\label{Ex: Different coordinates}
	Continuing with the previous example, $f(z) = z^2 + 1/t$, we can enlarge the vertex set $\Gamma'$ in order to obtain further information about the location of the mass of $\mu_f$. Looking at the Newton polygon of $z^2 - z + 1/t$, we see that $f$ has (classical) fixed points $p_{\pm}$ with $|p_{\pm}| = 1 / |\sqrt{t}|$, and they are easily seen to be repelling. Let us define
		\[
			\Gamma'' = \{\zeta_g, \; \zeta_{0, 1/|t|}, \;\zeta_{0, 1/|\sqrt{t}|}\}.
		\]
As $f(\zeta_{0,1/|\sqrt{t}|}) = \zeta_{0, 1/|t|}$, the pair $(f, \Gamma'')$ is also analytically stable. Let $V_{\pm}$ be the open disks with boundary point $\zeta_{0, 1/|\sqrt{t}|}$ and containing $p_{\pm}$, respectively. The set of states in this case is 
		\[
			\JJ'' = \{ V_+, \;, V_-, \; \zeta_g, \; \zeta_{0, 1/|t|},  \; \zeta_{0, 1/|\sqrt{t}|}\},
		\]
and the transition matrix $P''$ is given by 
	\[
		\bordermatrix{ 
			& V_+ & V_- & \zeta_g  & \zeta_{0, 1/|t|} & \zeta_{0, 1/|\sqrt{t}|} \cr
			V_+ &  1/2 & 1/2 & 0 &  0 & 0\cr
			V_- & 1/2 & 1/2 & 0 & 0  & 0\cr
			\zeta_g &  1/2 & 1/2 & 0 & 0  & 0 \cr
			\zeta_{0, 1/|t|} & 0 & 0 & 0 & 0  & 1 \cr			
			\zeta_{0, 1/|\sqrt{t}|} & 1/2 & 1/2 & 0 &  0 & 0\cr
		}		
	\]	
Now we have $(P'')^n = \mathbf{1}\nu$ for $n \geq 2$, where $\nu = \begin{pmatrix} 1/2 & 1/2 & 0  & 0 & 0 \end{pmatrix}$, so that 
	\[
		\mu_f(V_+) = \mu_f(V_-) = 1/2.
	\]

\subsection{The failure of Theorem~D without suitable hypotheses}
\label{Ex: Tent map}
The existence of analytically stable augmentations does not hold for arbitrary pairs $(f,\Gamma)$ over arbitrary non-Archimedean fields, even in residue characteristic zero. For example, suppose that $f$ is a Latt\`es map with non-simple reduction. Then the Julia set of $f$ is an interval in $\Berk$ and $f$ acts on the Julia set by a tent map. (See \cite[\S5.1]{Favre_Rivera-Letelier_Ergodic_2010}.) Let $\zeta$ be a Julia point with infinite orbit in the Julia set. By passing to an algebraically closed and complete extension of $\LL$, we may take $\zeta$ to be a type~II point. 
	
Let $\Gamma$ be any vertex set containing $\zeta$.  Every $\Gamma$-domain $U$ that meets the Julia set will be a $J$-domain (Proposition~\ref{Prop: Fatou F-ensemble}).  As the orbit of $\zeta$ is infinite, it must intersect a $J$-domain.  Hence $(f, \Gamma)$ cannot be analytically stable.

\subsection{The computations of \cite{DeMarco_Boundary_Maps_2005}}
\label{Ex: I(d)}
	The results in the present paper may be viewed as a natural generalization of \cite{DeMarco_Boundary_Maps_2005}.  The space of rational maps of degree $d\geq 2$ on the Riemann sphere $\Chat$ sits inside $\Ratbar_d = \P^{2d+1}$, the space of all pairs $(F,G)$ where $F,G \in \CC[X,Y]$ are homogeneous of degree~$d$.  We set $H = \gcd(F,G)$, so that $(F,G) = H\cdot\phi$, where $\phi$ describes an endomorphism of $\Chat$.   The iteration map $f\mapsto f^n$, from $\Ratbar_d$ to $\Ratbar_{d^n}$, is intedeterminate along a subvariety  $I(d) \subset \del\Rat_d$, independent of $n\geq 2$; the set $I(d)$ consists of all elements $H\cdot \phi$ for which $\phi$ is constant and its value is a root of the polynomial $H$. 

Let $\mu_f$ denote the measure of maximal entropy for the rational function $f: \Chat\to\Chat$.  The main result of \cite{DeMarco_Boundary_Maps_2005} states that the map of measures, $f\mapsto \mu_f$, extends continuously from $\Rat_d$ to a boundary point $H\cdot\phi \in \del\Rat_d$ if and only if $H\cdot \phi$ lies outside of $I(d)$.

\begin{prop}
Let $f_t$ be a degenerating 1-parameter family of rational functions of degree $d \geq 2$, and write $f$ for the associated Berkovich dynamical system. Then the limit $f_0$ lies outside $I(d)$ if and only if the pair $(f,\{\zeta_g\})$ is analytically stable.  
\end{prop}
	
\noindent
The formulas in \cite{DeMarco_Boundary_Maps_2005} for the mass of the limit measure $\mu_0 = \lim_{t\to 0} \mu_{f_t}$ may be deduced from our Markov chain description in Theorem~C.

\begin{proof}
Write $f_0 = H\cdot\phi \in \del \Rat_d$.  Then $f(\zeta_g) = \zeta_g$ on $\Berk$ if and only if the reduction map (given by the limit function $\phi$) is nonconstant.  If $\phi$ is constant, then the $d$ roots of $H$ coincide with the $d$ disks in $\Berk$ containing the preimages of $\zeta_g$, counted with multiplicities.  (Compare \cite[\S3]{Faber_Berk_RamI_2013}.)  The constant value of $\phi$ coincides with the Berkovich disk $U$ containing the image $f(\zeta_g)$.  Consequently, $f_0 \in I(d)$ if and only if $U$ contains a preimage of $\zeta_g$, and $(f, \{\zeta_g\})$ fails to be analytically stable.  On the other hand, if $\phi$ is constant but $f_0 \not\in I(d)$, then $f(U) \subset U$, so that $U$ is an $F$-disk and $(f,\{\zeta_g\})$ is analytically stable.  
\end{proof}


\appendix
\section{Rivera Domains, by Jan Kiwi}
\label{Sec: Appendix}


\medskip

Let $k$ be a characteristic zero  algebraically closed and complete non-Archimedean field with respect to a nontrivial absolute value $| \cdot |$. 
Denote by $\pberl$ the Berkovich projective line over $k$. Throughout this appendix Berkovich type I points will be also called {\bf rigid points}. 

Let  $f : \pberl \to \pberl$ be a rational map of degree at least $2$.  By definition, a point $\zeta \in \pberl$ belongs to the Julia set $\julia(f)$ if for every neighborhood $U$ of $\zeta$, the union of images $\bigcup_n f^n(U)$ omits at most finitely many points of $\pberl$.  The Fatou set is the complement of the Julia set.  
A connected component of the Fatou set will be simply called a {\bf Fatou component}.
Each Fatou component maps onto a Fatou component under $f$. 
A Fatou component $U$ is called {\bf fixed} if $f(U)=U$.
An attracting type I  fixed point $x_0$ (i.e.,  a fixed point $x_0$ such that $| f'(x_0)| < 1)$ belongs to the Fatou set. 
The component $U$ that contains $x_0$ is  fixed and it is called the {\bf immediate basin} of $x_0$.

The {\bf ramification locus of $f$} is formed by all $x \in \pberl$ such that the local degree $m_f(x)$ satisfies $m_f(x) \ge 2$.
We say that $f: \pberl \to \pberl$ is {\bf tame} if its ramification locus is contained in the convex hull of the rigid critical points of $f$.

In his doctoral thesis Rivera-Letelier \cite[Th\'eor\`eme 3]{Rivera-Letelier_Asterisque_2003} classified periodic Fatou components of rational maps for $k= \C_p$. The aim of this appendix is to provide a proof of this classification for the case of tame rational maps.

\begin{thm} [Rivera-Letelier]
  \label{thr:1}
    Let $f: \pberl \to \pberl$ be a tame rational map of degree $d \ge 2$. If $U$ is a fixed Fatou component, then exactly one of the following holds:
  \begin{enumerate}
  \item $U$ is the immediate basin of attraction of a type I fixed point. 
  \item The map $f: U \to U$ is a bijection and  $\partial U$ is a union of
at most $d-1$ type II periodic orbits.
  \end{enumerate}
\end{thm}

The proof relies on the general topological fact that maps such as $f$ must have a fixed point $x_0$ in $U$ (see subsection~\ref{ss:FixedPoint}). When the fixed point is an attracting rigid (type I) point we are in case (1). Otherwise, the local dynamics at non-rigid Fatou fixed points (see subsection~\ref{ss:FatouPeriodicPoints}) allows us to spread the periodic behavior from $x_0$ to the directions containing boundary points of $U$. In order to be able to reach the boundary of $U$ with this periodic behavior we employ some properties of injective maps (see subsection~\ref{ss:InjectiveMaps}). Finally, a counting argument establishes that $\partial U$ is finite and we conclude that $f:U \to U$ is a bijection.

To ease notation, in the sequel we drop the subscript $k$.
We identify an element  $\vec{v}$ in $T \pber_x$ with the corresponding open disk
$D=D(\vec{v})$ defining $\vec{v}$. 
Thus, we abuse notation and regard simultaneously 
$D$ as an element of  $T \pber_x$ and as a subset of $\pber$.
For short, we say that $D$ is a {\bf direction} at $x$.
Also, we let $\H = \pber \setminus \pone(k)$.

\subsection{Periodic points}
\label{ss:FatouPeriodicPoints}
Theorem~2.1 in~\cite{Kiwi_Quadratic_Puiseux_2014} says:

\begin{prop}
  \label{pro:1}
  Let $g$ be a rational map over $k$ of degree at least $2$, and let $x \in \H$ be a fixed point.
The fixed point $x$ belongs to the Julia set $\julia(g)$ if and only if one of the following occurs:
\begin{enumerate}
\item $m_g(x) \ge 2$ where $m_g(x)$ denotes the local degree of $g$ at $x$.

\item There exists a direction $D$ at $x$ with infinite forward orbit under $T_x g : T \pber_x  \to T \pber_x $ such that
$g(D) = \pone$. 
\end{enumerate}
\end{prop}

\begin{cor}
  \label{cor:1}
  If $x \in \H$ is a Fatou fixed point of $g$ and $D$ is a direction at $x$ with infinite forward orbit under $T_x g$, then 
$D$ is contained in the Fatou set of $g$.
\end{cor}

\begin{proof}
From Proposition~\ref{pro:1},  it follows that at a Fatou fixed point $x$ every direction $D$ with infinite forward orbit under $T_xg$ 
has zero surplus multiplicity, that is, $g(D)$ is a direction at $x$. 
In particular, $g^n(D)$ omits $T_xg^{-1}(D)$ for all $n \ge 0$. Thus, $D$ is contained in the Fatou set.
\end{proof}

\subsection{Injective maps}
\label{ss:InjectiveMaps}

For convenience we identify $\pone = \pone(k)$ with
$k \cup \{ \infty \}$ via the map $[z:1] \mapsto z$. 
As usual we denote by $\tilde{k}$ the residue field of $k$ (i.e. the ring of integers $\mathfrak{O}=\{ z \in k : |z|\leq 1\}$ modulo its maximal ideal $\mathfrak{M}=\{ z \in k : |z| < 1\}$).

The {\bf diameter} $\diam B$ of a Berkovich closed ball $B \subset \pber \setminus \{\infty\}$   
is by definition the diameter of $B \cap k$ with respect to the metric in $k$ induced by $| \cdot |$. For all $x \in \pber$ such that $x \neq \infty$
we set
$$ \diam x := \inf  \{ \diam B \ : \  x \in B, B \text{ closed} \}.$$

Recall that for $z \in k$ the Berkovich open (resp. closed) ball 
of diameter $r \in |k^\times|$ containing $z$ is denoted by $\mathcal{D}(z,r)^-$ (resp. $\mathcal{D}(z,r)$). The $\sup$ norm on $\{w \in k \ : \ |w-z|\le r \}$ is the unique boundary point of these Berkovich balls.

\begin{lem}[Rivera's approximation Lemma]
  \label{approximation}
Let $\z_g \in \pber$ denote the Gauss point
and $D= \mathcal{D}(0,1)^-$ the Berkovich open unit ball containing the origin.
Consider a rational map $g$ that fixes the Gauss point $\z_g$ and
such that $T_{\z_g} g (D) =D$.
Assume that for a closed ball $B$ contained in $D$ we have that $g: D \setminus B \to D$ is injective. 
Then there exists an injective analytic map $h : D \to D$ such that
$h(x) = g(x)$ for all $x$ with $\diam x \ge \diam B$. 
\end{lem}

\begin{proof}
Lemme d'Approximation in Section 5 of~\cite{Rivera-Letelier_Asterisque_2003}.
\end{proof}

\begin{lem}[Constant tangent map]
  \label{lem:2}
  Let $D$ be the Berkovich open unit ball containing the origin and
let $\frak{M} = D \cap k$ denote the maximal ideal. 

Let $g : D \to D$ be a bijective analytic map.
For all  rigid points $z_0$ and $\rho$  in $D$ we have that
$$g(z_0 + \rho z + \rho \frak{M}) = g(z_0) +  g'(0) \cdot (\rho z) + \rho \frak{M},$$
for all $|z| \le 1$.
\end{lem}

In other words, for \emph{all} $x \in D$ with diameter $r= |\rho| <1$,
the tangent map  $T_x g$ is multiplication by $\lambda=\widetilde{g'(0)} \in \tilde{k}$, 
in the ``coordinates'' of $T \pber_x$ and $T\pber_{g(x)}$  determined by the choice of $z_0$ and $\rho$.  These coordinates assign $\tilde{z} \in \tilde{k}$
 to the direction
$z_0 + \rho z + \rho \frak{M}$ at $x$  and  $\tilde{w} \in \tilde{k}$ to the direction
$g(z_0) + g'(0) \, (\rho w) + \rho \frak{M}$ at $g(x)$.

\begin{proof}
  Write the series of $g$:
$$ g(z) = a_0 + a_1 z + a_2 z^2 + \cdots. $$
It follows  that $|a_0| <1, |a_1|=1, |a_j| \leq 1$ for all $j \ge 2$ since $g:D \to D$ is bijective.
For all $|z| \leq 1$, 
$$\rho^{-1} ( g(z_0 + \rho z) - g(z_0) ) = a_1 z + a_2(2 z z_0 + \rho z^2) + \cdots$$
which is congruent    to  $a_1z = g'(0) z$, modulo  $\frak{M}$.
\end{proof}

Lemma~2.14 in~\cite{Rivera-Letelier_Asterisque_2003} in the language of Berkovich spaces reads as follows:
 
\begin{lem}[Injectivity domain]
  \label{lem:3}
  Let $g : \pber \to \pber$ be a rational map.
Let $V$ be a  connected component of 
$$\pber \setminus \{ x \in \pber \ : \   m_g(x) \ge 2 \}.$$
Then $g: V \to g(V)$ is a bijection. 
\end{lem}

\subsection{Fixed point}
\label{ss:FixedPoint}

\begin{lem}
  \label{lem:4}
   Let $f: \pber \to \pber$ be a tame rational map of degree at least $2$. If $U$ is a fixed Fatou component,  then $U$ contains a rigid attracting fixed point or a type II fixed point.
\end{lem}

\begin{proof}
Note that $f: \overline{U} \to \overline{U}$ is a continuous self-map of a compact, Hausdorff acyclic and locally connected tree in the sense of Wallace~\cite{Wallace_Fixed_Point}. Thus it has a fixed point $x_0$.

Assume that $U$ contains no rigid attracting fixed point. We must conclude that 
$U$ contains a type~II fixed point.

If $x_0$ is an indifferent rigid point (i.e., $|f'(0)|=1$), then there exist arbitrarily small rigid closed balls around it which are fixed under $f$. 
The associated Berkovich type II points are fixed and we may choose one 
of these points in $U$.

If $x_0$ is a repelling rigid point, then let $V$ be a small Berkovich open  ball about $x_0$ such that
$f(V)$ compactly contains $V$. Denote by $x_V$ the boundary point of $V$.
Consider the continuous map $F: \overline{U} \setminus V
\to \overline{U} \setminus V$ defined by $F(x) =f(x)$ if $f(x) \notin f(V)$ and 
$F(x)=f(x_V)$ otherwise. It follows that $F$ has a fixed point which is not in $V$ and
hence it is a fixed point of $f$. 

Since there are at most finitely many rigid fixed points, we may now  assume that  there exists $x_0 \in \partial U \cap \H$ which is a fixed point of $f$.
(Although not needed below, we observe that, from Proposition~\ref{pro:1}, the fixed point $x_0 \in \julia(f)$ is of type II.)
It follows that the direction of $U$ at $x_0$ is fixed  under $T_{x_0} f$.
If the degree in this direction is $1$, then, as in the rigid indifferent case, there exists a type II fixed point in that direction. 
If the degree in that direction is $\ge 2$ we proceed as in the rigid repelling case and remove a small Berkovich open ball. Since the boundary  of an
 open and connected set $U$ has at most finitely many intersections with the ramification locus of $f$ (points of degree at least $2$ in $\pber$), after finitely many
removals we obtain a type II fixed point in $U$. 
\end{proof}

\subsection{Proof of Theorem~\ref{thr:1}}
We work under the hypothesis of Theorem~\ref{thr:1} and 
 assume that $U$ is not the immediate basin of attraction of a rigid fixed point.
We must prove that (2) in the statement of the theorem holds.

By Lemma~\ref{lem:4}, we may assume that the Gauss point $\z_g$ is a fixed point 
which lies in $U$. From Proposition~\ref{pro:1} 
the local degree of $f$ at $\z_g$ is $1$ (i.e., $m_f(\z_g)=1$).

\begin{lem}
  \label{lem:5}
  Let  $U$ be a fixed Fatou component and assume that $\z_g \in U$ is a fixed point. For all $y_0 \in \partial U$, 
  the arc $[y_0,\z_g]$ joining $y_0 \in \partial U$ and $\z_g$ is periodic.
That is, there exists $m \ge 1$ such that $f^m (y) =y$ for all $y \in [y_0,\z_g]$.
\end{lem}

\begin{proof}
Changing coordinates we may assume that $y_0$ lies in the Berkovich open unit ball 
$D$ containing the origin.
By Corollary~\ref{cor:1} and passing to an iterate of $f$, 
we may suppose that the direction $D$ 
is fixed under $T_{\z_g} f$. 

Let $$ r = \inf \{ \diam y \ : \  y \in [y_0,\z_g],  [y,\z_g] \mbox{ is periodic under } f \}.$$
The lemma follows from continuity of $f$ once we show that $r = \diam y_0$ 
and that the period of $y \in ]y_0,\z_g]$ has an uniform upper bound $m$. 

Let $y_0' \in [y_0,\z_g]$ be such that $\diam y_0'=r$.  
Note that for all $y \in ]y'_0, \z_g]$ we have that $[y,\z_g]$ is a periodic interval, say of period $p$. That is, 
$f^p: [y,\z_g] \to [y,\z_g]$ is the identity. Moreover, $m_{f^p} (z) =1$ for all $z \in [y,\z_g]$, for otherwise $z$ would be a Julia periodic point, by
Proposition~\ref{pro:1}.

Let $\Gamma$ be the convex hull of the set obtained as the union of 
the ramification locus of $f$
and the Gauss point $\z_g$.
By tameness, $\Gamma$ is contained in the the convex hull of the set formed by the rigid critical points of $f$ and $\z_g$.
Consider $R$ such that $r< R <1$ and  no vertex of $\Gamma$ has diameter in $]r,R]$.

Let $x_0' \in [y_0', \z_g]$ be the  point of diameter $R$. Observe that $x_0'$ is periodic under $f$, say of period $p_0'$. 
By Corollary~\ref{cor:1}, the direction of $y_0'$ at $x_0'$ is periodic under $T_{x_0'}f^{p'_0}$ since it contains $y_0 \in \julia(f)$.  Without loss of generality we assume
that this direction is fixed under $T_{x_0'}f^{p'_0}$. 
Since $f^{p'_0}:  [x_0', \z_g] \to [x_0', \z_g] $ is the identity, the direction of $\infty$ is also fixed. Thus, in an appropriate coordinate of $T \pber_{x_0'} \equiv \tilde{k} \cup \{ \infty \}$
 $T_{x_0'}f^{p'_0}$ is $z \mapsto \lambda z$ for some $\lambda \in \tilde{k}$ where $0$ corresponds to the direction $D_0$ of $y_0'$ at $x_0'$ and $\infty$ corresponds to the direction of $\z_g$.

Let $x'_n = f^n(x'_0)$ and  $D_n =T_{x_0'}f^{n} (D_0) \in T_{x_n'} \pber$, subscripts $\pmod{p_0'}$. Note that $[x_n',\z_g]$ maps isometrically onto  $[x_{n+1}',\z_g]$ since it is a periodic interval. Thus, $D_n$ is not the direction of $\infty$ and $\diam x'_n = R$ for all $n$. 
If $\Gamma \cap D_n \neq \emptyset$, then there exists a unique element $w$ of  $\Gamma \cap D_n$ of diameter $r'$ for all $r \le r' \le R$ (by the choice of $R$).
Let $B_n$ be the closed (possibly degenerate) ball contained in $D_n$ with boundary point being the unique point in $\Gamma \cap D_n$ of diameter $r$. If $\Gamma \cap D_n = \emptyset$, then let $B_n$ be any 
ball of diameter $r$ in $D_n$.  
By the choice of $R$ the ramification locus is either disjoint from $D_n \setminus B_n$ 
or its intersection with $D_n \setminus B_n$ is the open interval joining the  boundary point of $ B_n$ with  $x'_n$.
The latter is impossible since the multiplicity at $x_n'$ is $1$. By Lemma~\ref{lem:3},
we conclude that  $f: D_n \setminus B_n \to D_{n+1}$ is injective.
By Rivera's approximation Lemma~\ref{approximation}, there exists an analytic bijection
$g_n : D_n \to D_{n+1}$ which agrees with $f$ for all $y \in D_n$ such that $\diam y \ge r$. 

Consider any $y \in ]y_0' , x_0']$. Since $[y,\z_g]$ is a periodic interval, $\diam f^n (y) = \diam y > r$ for all $n$. 
In particular, $f^n(y) \in D_n \setminus B_n$ for all $n$ $\pmod{p_0'}$. 
Hence $y$ is also periodic under $G=g_{p_0'-1} \circ \cdots \circ g_0$. 

The direction of $y_0'$ at $x_0'$ is fixed under $G$. Thus,  $[y,x_0'] \cap G([y,x_0'])= [x_1, x_0']$ where $x_1$ is fixed under 
$G$. If $x_1 \neq y$, then $T_{x_1} G$ has finite order $>1$, since $y$ is contained in a periodic direction which is not fixed and $\z_g$ is in a fixed direction.
Say that the order of $T_{x_1} G$ is $q$. From the constant tangent map Lemma~\ref{lem:2}, it follows that $q p_0'$ is the order of $\lambda$.
Now $[y,x_1] \cap G^q ([y,x_1]) = [x_2,x_1]$. It follows that 
$x_2$ is fixed under $G^q$ and $x_2 \neq x_1$. If $x_2 \neq y$, then $T_{x_2} G^q$ has finite order $>1$ which is a contradiction with the fact that $\lambda$ has order $q p_0'$ and the constant tangent map Lemma~\ref{lem:2}. Therefore $x_2 = y$ and the period of $y$ is the order of $\lambda$.
We have shown that for all $y \in ]y_0' , x_0']$, the period of $y$ is bounded above by $p_0'$ if $x'_1 = y$ and  by the order of $\lambda$ otherwise. 
By continuity, it follows that $[y_0',x_0']$ is a periodic interval.

Since $y_0'$ is periodic, from Corollary~\ref{cor:1} we have that 
the direction of $y_0$ at $y_0'$ is periodic. Therefore $y_0 = y_0'$, for otherwise
we may find a periodic sub-interval of $[y_0, y_0']$ including $y_0'$ which contradicts the definition of~$r$. 
\end{proof}

To finish the proof of Theorem~\ref{thr:1} we first show that $U$ has finitely many boundary points, all periodic and all of type II and then proceed to show that
$f:U \to U$ is a bijection.

From the previous lemma we have that every point $y_0 \in \partial U$ is a periodic point in the Julia set. Hence, it is a rigid point or a type II point.
However, $y_0$ is not a rigid point, for otherwise $y_0$ would be a repelling periodic point. From the previous lemma $[y_0,x_0]$  would be 
a periodic interval, which would be incompatible with the repelling nature of $y_0$.
Thus, $y_0$ is a type II point.

Let $\cO = \{y_0, \dots, y_{p-1}\}$ be the orbit of $y_0 \in \partial U$. 
Let $B_0, \dots, B_{p-1}$ be the complement of the direction of $U$ at $y_0, \dots, y_{p-1}$, respectively.
Then, at least one of these balls 
$B_j$ must map onto $\pber$ (otherwise all $B_j$ and therefore $\partial B_j$ would be contained in the Fatou set). 
Denote such a ball by $B(\cO)$.
For distinct periodic orbits $\cO_0, \cO_1$ contained in $\partial U$,
 the corresponding closed balls 
$B(\cO_0)$ and $B(\cO_1)$ are disjoint and each contains a preimage of $U$.
Therefore, $\partial U$ consists of at most $d-1$ type II periodic orbits. 

It remains to show that $f: U \to U$ is a bijection.
By Lemma~\ref{lem:3} it is sufficient to prove that the ramification locus is disjoint from $U$. 
We proceed by contradiction and let $\gamma \subset U$ be a connected component
of the ramification locus. 
The convex hull $\Gamma_U \subset \overline{U}$ of $\partial U$ is fixed pointwise by an iterate of $f$. Passing to this iterate we may
assume that $\Gamma_U$ is pointwise fixed. Thus, $\Gamma_U \cap \gamma = \emptyset$ since $\Gamma_U \setminus \partial U$ consists of Fatou fixed points.
Let $a \in \gamma$ and $b \in \Gamma_U$ be such that $]a,b[ \subset \pber$ is an  arc disjoint from $\gamma \cup \Gamma_U$.
Without loss of generality, replacing $\gamma$ if necessary, we may assume that the multiplicity of $f$ along $]a,b[$ is $1$.
It follows that $f(]a,b]) = ]f(a),b]$ and the direction of $b$ at $a$ maps under   $T_a f$ onto the direction $D_b$ of $b$ at $f(a)$ with multiplicity $1$. 
Since the multiplicity at $a$ is at least $2$ there exists a direction 
$D_a$ at $a$ not containing $b$ which is also mapped under $T_a f$ onto the direction $D_b$ of $b$ at $f(a)$. 
The connected graph $\Gamma_U$ is disjoint from $D_a$ and contained in $D_b$, therefore $D_a \subset U$ which implies that $f(D_a) \subset U$.
But $f(D_a) \supset D_b \supset \Gamma_U \supset \partial U$  which is a contradiction.
This completes the proof of Theorem~\ref{thr:1}.

%


\bibliographystyle{plain}
\bibliography{xander_bib}

\end{document}